\documentclass[12pt]{article}
\usepackage{tikz}
\usepackage[left=2.5cm,right=2.5cm,top=2.5cm,bottom=2.5cm]{geometry}
\usepackage{cmap}
\usepackage[T2A]{fontenc}
\usepackage[utf8]{inputenc}
\usepackage{amssymb, amsmath, amsthm}
\usepackage{epsfig,graphicx,float}
\usepackage{physics}
\usepackage{hyperref}
\usepackage{authblk}
\usepackage{wasysym}

\usepackage[autostyle, english=american]{csquotes}
\MakeOuterQuote{"}

\DeclareMathOperator{\conv}{conv}
\DeclareMathOperator{\Rg}{Rank}
\DeclareMathOperator{\Ker}{Ker}
\DeclareMathOperator*{\argmin}{arg\,min}
\newcommand{\R}{\mathbb{R}}
\newcommand{\Cc}{\mathbb{C}}

\theoremstyle{definition}
\newtheorem{theorem}{Theorem}[section]
\theoremstyle{definition}

\newtheorem{proposition}[theorem]{Proposition}

\theoremstyle{definition}
\newtheorem{definition}{Definition}[section]

\everymath{\displaystyle}

\title{Geometry of quadratic maps via convex relaxation}

\author[1,2]{Anatoly Dymarsky}
\author[1,3]{Elena Gryazina}
\author[1,4]{Sergei Volodin}
\author[3]{Boris Polyak}
\affil[1]{Skolkovo Institute of Science and Technology}
\affil[2]{University of Kentucky}
\affil[3]{Institute for Control Sciences RAS}
\affil[4]{\'Ecole Polytechnique F\'ed\'erale de Lausanne}

\begin{document}
\maketitle

\begin{abstract}
We consider several basic questions pertaining to the geometry of image of a general quadratic map.
In general the image of a quadratic map is non-convex, although there are several known classes of quadratic maps when the image is convex.
Remarkably, even when the image is not convex it often exhibits hidden convexity -- a surprising efficiency of convex relaxation to address various geometric questions by reformulating them in terms of convex optimization problems.
In this paper we employ this strategy and put forward several algorithms that solve the following problems pertaining to the image: verify if a given point does not belong to the image; find the boundary point of the image lying in a particular direction; stochastically check if the image is convex, and if it is not, find a maximal convex subset of the image. 
Proposed algorithms are implemented in the form of an open-source MATLAB library \href{http://github.com/sergeivolodin/CAQM}{CAQM}, which accompanies the paper.
Our results can be used for various problems of discrete optimization, uncertainty analysis, physical applications, and study of power flow equations.

\end{abstract}

\section{Introduction}
\label{intro}
In this paper we discuss geometric properties of images of general real-valued quadratic maps.
Full image of a quadratic map is an unbounded set in $\R^m$ with its boundary being an appropriate real algebraic variety.
There are several basic questions pertaining to the geometry of quadratic maps which we address below.
First question is the feasibility of a given point, i.e.~if a particular point in $\R^m$ belongs to the image of a given quadratic map.
Second question is to identify a point on the boundary of the image that would lie on a given ray in $\R^m$.
Third question is to verify if the full image is convex, and, if not, to identify a maximal possible convex subset within it.

These and related questions are of obvious practical importance.
They naturally arise in the problems of discrete optimization \cite{luo2010semidefinite, Zhang}, uncertainty analysis \cite{packard1993complex}, and problems related to Power Flow study \cite{LavaeiLow}.
In particular, discrete optimization over a boolean variable $x \in \{-1, 1\}$ can be reduced to a continious case using quadratic constraint $x^2 = 1$.
Similarly, in control theory, the $\mu$-based methods (so-called $\mu$-analysis and synthesis) have proved  useful for the performance analysis of linear feedback systems under uncertainty \cite{packard1993complex}. In this case the quantity of interest is the structured singular value $\mu$. It is easy to calcualte an upper bound on $\mu$ via convex optimization, but the latter becomes exact whenever the corresponding quadratic map is convex \cite{Doyle96}, \cite{zhou1996robust}.

The geometric problems outlined above are usually difficult to solve. In fact, some of these  problems are known to be NP-hard \cite{Ramana_NPhard94}.
Hence it is highly desirable to develop theoretical and numerical approaches which may rely on peculiarities of a particular formulation and yield an efficient, if not universal, tool to address these questions.
In general the image of a quadratic map is non-convex, although there are a few known classes of quadratic maps with convex images.
Nevertheless often quadratic maps exhibit "hidden convexity" which can be understood heuristically as an unexpected efficiency of various convex relaxations.
Sometimes this efficiency can be justified theoretically \cite{luo2010semidefinite}. 

One of the important geometric notions which we employ and further develop in this paper is of boundary non-convexity \cite{Dymarsky_cuttingArX}.
Combining it with the ideas of convex relaxation and Linear Matrix Inequality (LMI) we formulate a number of algorithms to address the questions outlined above, as well as some other mathematical problems, which are of interest in their own right.
The algorithms proposed in this paper are implemented in an open-source MATLAB library \href{http://github.com/sergeivolodin/CAQM}{Convex Analysis of Quadratic Maps (CAQM)}, which accompanies the paper.





\label{known_facts}




\section{Notations}
\label{Notations}
We start with the definition of a quadratic map.
\begin{enumerate}
\item Real case, the map $f\colon \mathbb{R}^n\to\mathbb{R}^m, f=(f_1,\dots, f_m)$
\begin{equation}
f_k(x)=x^TA_k x+2b_k^Tx ,\quad A_k=A_k^T ,\quad x, b_k\in \mathbb{R}^n ,\quad k=1\dots m . \label{real}
\end{equation}

\item Complex case, the map $f\colon \mathbb{C}^n\to\mathbb{R}^m$
\begin{equation}
f_k(x)=x^*A_k x+b_k^*x+x^*b_k ,\quad A_k=A_k^* ,\quad x, b_k\in \mathbb{C}^n ,\quad k=1\dots m ,\label{complex}
\end{equation}
\end{enumerate}
where $\cdot^*$ stands for a Hermitian conjugate.
We will use $\mathbb{V}$ in what follows to denote $\mathbb{R}^n$ or $\mathbb{C}^n$ depending on the context.
With some exceptions both cases will be treated in parallel, as most results equally apply to both real and complex $\mathbb{V}$.
By default we will assume complex case, and will specify when the real case should be treated differently.
Another related comment is that a complex map $f:{\mathbb C}^n \rightarrow {\mathbb R}^m$ can be trivially re-written as a real map $f:{\mathbb R}^{2n} \rightarrow {\mathbb R}^m$.
Although this would lead to exactly the same results in certain cases, there is an important difference between these two representations, which is discussed after Proposition \ref{th:noconv_cert}.

The geometric questions we are interested in are independent of the affine transformations of $x$ and $y=f(x)$.
That allows us to choose $f(x)$ in \eqref{real} and \eqref{complex} such that $f(0)=0$.
Furthermore shifting $x\rightarrow x-x_0$ also shifts $b_k\rightarrow b_k-A_k x_0$.
By saying that $b_k$ {\em is} or {\em is not} trivial we would emphasize that the system of linear equations $A_k x_0=b_k,\,k=1\dots m$, has or does not have a solution $x_0$.

It is convenient to introduce a standard Euclidean scalar product in $\mathbb{R}^m$, such that for two vectors $c,\,y\in \mathbb{R}^m$, $c\cdot y=\sum_{k=1}^m c_k\, y_k$.
To simplify the notations we extend that definition to a case when one of the arguments is a tensor. \\[-20pt]
\theoremstyle{definition}
\begin{definition}
For a vector $c=(c_1,...,c_m)$ and a tuple of vectors $b=(b_1,...,b_m), \ b_k \in \mathbb{V}$, or a tuple of $n\times n$ matrices $A=(A_1,...,A_m), \ A_k\in \mathbb{R}^{n\times n}$ or $\mathbb{C}^{n\times n}$, the dot product is defined as follows,
\begin{eqnarray}
c\cdot b=\sum\limits_{k=1}^m c_k\, b_k ,\qquad
c\cdot A=\sum\limits_{k=1}^m c_k\, A_k\ . \nonumber
\end{eqnarray}
\end{definition}

The main object we are going to study is the full image $F$ of $f$. It can be defined as a set of points $y\in \mathbb{R}^m$ such that the system of quadratic equations $y=f(x)$ has a solution $x$. $F$ is a non-trivial subset in $\mathbb{R}^m$.
To emphasize this interpretation of $F$ we will also call it the {\it feasibility set}. \\[-25pt]
\begin{definition}
$F$ is the full image of $f$,
	$$F=f(\mathbb{V})=\{y\in \mathbb{R}^m|\, \exists\, x\in \mathbb{V},\, y=f(x) \}\subseteq \mathbb{R}^m\ .$$
\end{definition}

\begin{definition} $G$ is the convex hull of $F$,
	$$G=\conv (F)\subset \mathbb{R}^m\ .$$
\end{definition}

To investigate geometric properties of $F$ we will often study the intersection of $F$ with a supporting hyperplane, which is specified by a normal vector $c$. \\[-25pt]

\begin{definition} $\partial F_c $ is the set of boundary points of $F$ "touched" by a supporting hyperplane with the normal vector $c\in\mathbb{R}^m$,
	$$\partial F_c=\argmin\limits_{y\in F}(c\cdot y)\ .$$
\end{definition}
\begin{definition} $\partial G_c$ is the set of boundary points of $G$ "touched" by a supporting hyperplane with the normal vector $c\in\mathbb{R}^m$,
	$$\partial G_c=\argmin\limits_{y\in G}(c\cdot y)\ .$$
\end{definition}
A priori a supporting hyperplane orthogonal to $c\in\mathbb{R}^m$ may not exist, in which case $\partial F_c$ and $\partial G_c$ would be empty.
There is a particular class of quadratic maps, which we, following \cite{Sheriff}, will call {\it definite}. For such maps there exists at least one vector $c\in \mathbb{R}^m$ such that $c\cdot A\succ 0$.\\[-25pt]
\begin{definition}\label{ex:c_plus}
	The set of all vectors $c\in \mathbb{R}^m$, such that $c\cdot A\succcurlyeq 0$ is denoted as ${\mathcal K}$, and
	$${\mathcal K}_+=\{c\in\R^m\,\big|\,c\cdot A\succ 0 \}=\mathcal K\setminus\partial \mathcal K\ .$$
\end{definition}
\vspace{-.2cm}
\noindent
The set ${\mathcal K}_+$ is a cone, and the position of $c$ within ${\mathcal K}$ defines the spectrum of $c\cdot A$.
When the map is definite, ${\mathcal K}_+$ has dimension $m$.
It is easy to see that $\partial F_c$ is non-empty only when $c\in {\mathcal K}$.
The opposite is also true, modulo an important subtlety.
If the map is definite and $c\in {\mathcal K}$ but $c\notin \partial {\mathcal K}$, it is easy to see that $c\cdot A\succ 0$ and $\partial F_c$ would consist of exactly one point.
When $c\in \partial {\mathcal K}$, there are two possibilities: $\partial F_c$ could be empty, or could include an infinite number of points.
In the latter case $\partial F_c$ might be non-convex  -- this is boundary non-convexity, which implies non-convexity of $F$.
In our approach to test the convexity of $F$ we will be looking specifically for such directions $c$\\[-25pt]
\begin{definition}\label{ex:c_minus}
	Set of vectors $c\in \mathbb{R}^m$, such that $\partial F_c$ is non-convex is denoted as $C_{\rm ncvx}$:
$$C_{\rm ncvx}=\{c\in\R^m\,\big|\,\mbox{set }\partial F_c\mbox{ is non-convex}\}\ .$$
\end{definition}
\vspace{-.2cm}
\noindent It can be easily seen that for definite maps $C_{\rm ncvx}\subset \partial {\mathcal K}$.
Clearly, if $C_{\rm ncvx}$ is not empty the corresponding set $F$ is not convex.
The opposite is also true up to some technicality. Thus, it was shown in \cite{Gutkin,Sheriff} for homogeneous $b_k=0$ maps and in \cite{Dymarsky_SmallBall,Dymarsky_cuttingArX} for the general case that, up to some additional conditions and technical details, the absence of boundary non-convexities can be supplemented by a topological argument to establish convexity of $F$.
Hence identifying boundary non-convexities of quadratic maps is sufficient to verify the convexity of $F$.

\begin{definition}
For symmetric or hermitian matrices we introduce the standard scalar product $\langle X,Y\rangle$ = $\trace(X\,Y)$.
\end{definition}

\section{Geometry and the role of convexity}
The main idea of this paper is to reformulate various questions pertaining to the geometry of $F$ in form of the optimization problems.
When $F=G$ is convex, the corresponding optimization problems would be the problems of convex optimization which allow for an efficient numerical solution.
The starting point is a rather standard observation that $F$ can be formulated as an image of an auxiliary {\it linear} map.
\begin{theorem}[]
The image $F$ of $f$ is also an image of the following linear map with one additional non-linear constraint \cite{Ramana}
\begin{eqnarray}
\label{FX}
F &=& \left\{ \mathcal{H}(X)\,|\, X\succeq 0,\, X_{n+1,n+1} = 1,\, {\rm rank}(X)=1\right\}\ ,\\
\mathcal{H}(X)&=&(\langle{H}_1,X\rangle, \langle{H}_2,X\rangle, \dots,
\langle{H}_m,X\rangle)^T\ ,\quad H_{k} = \left(\begin{array}{cc} A_k & b_k \\ b_k^* & 0 \\\end{array}\right)\ .
\end{eqnarray}
\end{theorem}
\noindent where $X$ is a Hermitian $(n+1)\times(n+1)$ matrix $X = X^* \in \mathbb{C}^{(n+1)\times(n+1)}$ with entries $X_{ij}$. 
The condition $ {\rm rank}(X)=1$ is non-linear which makes the analysis of $F$ complicated, and the corresponding optimization problems non-convex.
Hence, the next crucial step is to substitute $F$ by its convex relaxation -- the convex hull $G$.
\begin{theorem}[]
Convex hull $G$ of $F$ is a convex relaxation of \eqref{FX} \cite{Ramana, PolyakGryazina17}
\begin{eqnarray}
\label{GX}
G = \text{conv}(F) = \{ \mathcal{H}(X)\,|\, X\succeq 0, X_{n+1, n+1} = 1\}\ .
\end{eqnarray}
The only difference between \eqref{GX} and \eqref{FX} is that the non-linear constraint ${\rm rank}(X)=1$ is removed.
Now $X$ only satisfies linear matrix inequality $X\succeq 0$ and an additional linear constraint $X_{n+1, n+1} = 1$ which makes the space of $X$ convex.
This is important as it allows to formulate various geometrical questions about $G$ in terms of convex optimization problems in the space of $X$.
As we will see shortly these optimization problems would often have a standard form, extensively discussed in the literature previously \cite{Boyd}.
\end{theorem}

Substituting $F$ with $G$ requires for $F$ to be convex, which is not always the case.
Nevertheless there are certain special cases, when $F$ is known to be convex.
One special class is the homogeneous maps $b_k=0$, or equivalently trivial $b_k$.
In this case convexity of image $F$ is closely related to the convexity of the image of a sphere. Indeed for the quadratic map $f$ we can introduce 
\begin{equation}
\label{Hdef}
H=\{y\in \mathbb{R}^m|\, \exists\, x\in \mathbb{V},\, |x|=1,\, y=f(x) \}\subseteq \mathbb{R}^m\ ,
\end{equation}
where $|x|$ stands for the Euclidean norm of $x$. The set $H$ is a a cross-section of the full image  $F\subset\mathbb{R}^{m+1}$ of the extended map $f=(f_1,\dots, f_m, f_{m+1}), f_{m+1}(x)=|x|^2$, with the hyperplane $y_{m+1}=1$. It is easy to see that convexity of $F$ implies the convexity of $H$ and vice versa. Similarly, for any definite quadratic map $f:{\mathbb R}^n\rightarrow {\mathbb R}^m$ convexity of $F$ can be reformualted as the convexity of image $H$ of an appropriate $(n-1)$-dimensional ellipsoid inside ${\mathbb R}^n$. Thus for homogenious maps it is  sufficient to consider convexity of the full image only.

For a few cases of homogenious $f$ listed below the convexity of $F$ and $H$ has been established analytically.
\begin{itemize}
\item If $m=2$, the map $f$ is homogeneous and $\mathbb{V}=\mathbb{C}^n$, then the image of the sphere $H$ \eqref{Hdef} is convex. This is a famous result by Hausdorff and Toeplitz \cite{Toeplitz,Hausdorff}. 

\item If $m=2$, the map $f$ is homogeneous and $\mathbb{V}=\mathbb{C}^n$, then $F$ is convex. This follows from the previous result by Hausdorff and Toeplitz. 

\item If $m=3$, the map $f$ is homogeneous and definite, and $\mathbb{V}=\mathbb{C}^n$, then $F$ is convex. This is also a corollary of the result by Hausdorff and Toeplitz.

\item If $m=2$, the map $f$ is homogeneous, and $\mathbb{V}=\mathbb{R}^n$, then $F$ is convex \cite{Dines}.

\item If $m=2$, the map $f$ is homogeneous, and $\mathbb{V}=\mathbb{R}^n$, $n\ge 3$ then the corresponding $H$ is convex \cite{Brickman}. 

\item If $m=3$, the map $f$ is homogeneous and definite and $\mathbb{V}=\mathbb{R}^n$, $n\ge 3$, then $F$ is convex \cite{Calabi,polyak98}.
Convexity of $F$ in this case is mathematically equivalent to the convexity of $H$ in the preceding case. 
  
\item If $m\ge 4$, the map $f$ is homogeneous and definite, satisfies a set of additional conditions, and $n\ge m$, then $F$ is convex \cite{Gutkin, Sheriff}.

\item If the map $f$ is homogeneous, $\mathbb{V}=\mathbb{R}^n$ with $n \geq 2$, and all matrices $A_i$ mutually commute, then $F$ is convex \cite{Fradkov}.

\item Some additional sufficient conditions for the convexity of $F$ for a homogeneous definite $f$ were formulated in \cite{Dymarsky_cuttingArX}.
\end{itemize}

When $b_k$ is non-trivial a few additional cases are known when $F$ is convex.

\begin{itemize}
\item If $m=2$, the map $f$ is definite and $\mathbb{V}=\mathbb{R}^n$, then $F$ is convex \cite{polyak98}.

\item If the map $f$ is definite and satisfies a set of additional conditions, which can be colloquially summarized as the absence of boundary non-convexities, with $n\ge m$, then $F$ is convex \cite{Dymarsky_cuttingArX}.
\end{itemize}

These criteria ensure that many quadratic maps which appear in practical applications are convex, e.g.~the solvability set of Power Flow equations for balanced distribution networks \cite{DymarskyTuritsyn}.
Moreover this list is likely to be incomplete, with many other maps $f$ which do not satisfy any of the aforementioned criteria have convex $F$.
The very practical complication here is that even if $F$ is convex, checking it for $m>2$ is NP-hard \cite{Ramana_NPhard94}.
One of the important results of this paper is a formulation of a stochastic algorithm which can detect and certify non-convexity of $F$ with a non-vanishing probability.
Hence running this algorithm for a sufficient time can ensure convexity of $F$ with almost complete certainty.

Another important observation is that even $F$ is not known to be convex, various optimization problems pertaining to $F$ can be very effectively solved in practice via convex relaxation, see e.g.~\cite{LavaeiLow}.
One possible explanation here is that although the full $F$ may not be convex, a subpart of it confined to a particular compact region which is important in the context of a particular application is convex.
A central result of this work is a numerical procedure which uses the stochastic algorithm mentioned above to identify a maximal compact subpart of $F$ which is likely to be convex.

Finally, we would like to mention that even when $F$ possesses no convexity properties, answering certain geometric questions about $G$ would suffice to establish a similar result about $F$.
Thus, establishing that a particular point $y\in \mathbb{R}^m$ does {\it not} belong to $G$ is also sufficient to show $y\notin F$. We formulate the algorithms which solve this and other problems below.

%
%
%

\section{Infeasibility certificate}
\label{certif}
In this and the next section we follow \cite{PolyakGryazina17}.
To check if a particular point $y^0\in \mathbb{R}^m$ is feasible, i.e.~belongs to $F$, we start with the analogous question for $G$.
The condition $y^0 \in G$ is equivalent to the following LMI being feasible, i.e.~the following system of (in)equalities admitting a solution,
\begin{equation}\label{eq:LMI1}
\mathcal{H}(X) = y^0, \quad X\succeq 0, \quad X_{n+1, n+1} = 1.
\end{equation}
Feasibility of this convex optimization problem can be verified efficiently \cite{Boyd}.
We prefer to formulate the same problem in dual terms.
If a point does not belong to a convex domain they can be always separated by an appropriate hyperplane.
This is illustrated in Fig.~\ref{fig:sep_hyp} below.
For a given vector $c\in \mathbb{R}^m$ we introduce the following matrix
\begin{equation}\label{eq:Hc} H(c) = \left(\begin{array}{cc} c\cdot A & c\cdot b \\ c\cdot b^* & -c \cdot y^0 \\ \end{array}\right).\end{equation}
\begin{figure}[htb]
\centerline{\includegraphics[width=5cm]{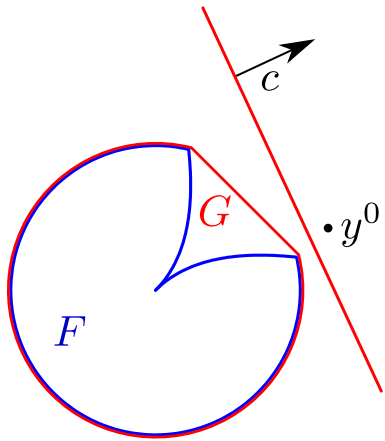}}
\caption{Infeasibility certificate via separating
hyperplane.}\label{fig:sep_hyp}
\end{figure}
\begin{theorem}[Sufficient condition of infeasibility]\label{th:suff_cond}
If for a given $y^0\in \mathbb{R}^m$ there exists $c\in \mathbb{R}^m$ such that $H(c)\succ 0$, then $y^0$ is infeasible with respect to $G$ and correspondingly with respect to $F$ \cite{PolyakGryazina17}.
\end{theorem}
\begin{proof}
Via Schur complement $H(c)\succ 0$ $\Leftrightarrow$ $c\cdot A\succ 0$ and $-c\cdot y^0 - (c\cdot b)^* (c\cdot A)^{-1} (c\cdot b) >0$.
But the latter inequality means
\begin{eqnarray}\nonumber
c\cdot y^0 < - (c\cdot b)^* (c\cdot A)^{-1} (c\cdot b) = \min_x x^* (c\cdot A) x + 2\Re(x^* (c\cdot b)) = \min\limits_{y\in F}\, (c\cdot y)\ .
\end{eqnarray}
The latter condition means there exists a separating hyperplane, defined by its normal vector $c$, that strictly separates $y^0$ and $G={\rm conv}(F)$. Hence $y^0$ does not belong to $F$.\\
{\bf Corollary.}
If $F$ is convex, the sufficient condition given by Theorem \ref{th:suff_cond} is also necessary.
In case $F$ is non-convex, even if the premise of the theorem fails and hence $y\in G$, it does not imply anything about $y^0\in F$. 
\end{proof}

The algorithm certifying infeasibility of $y$ with respect to $G$ and $F$ based on Theorem \ref{th:suff_cond} is implemented in the accompanying library as \underline{\tt infeasibility\_oracle.m}.

\section{Non-convexity certificate}
One of the central questions is to verify convexity of $F$.
This task requires several distinct steps, each being of interest in their own right. The presentation of this section follows \cite{PolyakGryazina17}.

\subsection{Boundary non-convexity}
\label{boundary_nonconvexity}
The underlying idea of certifying non-convexity of $F$ is to find vector $c\in\mathbb{R}^m$ such that the corresponding set $\partial F_c$ is non-convex.
The geometry of $\partial F_c$ depends on the spectrum of $c\cdot A$.
First, if $c\cdot A$ is positive-definite the corresponding supporting hyperplane intersects $F$ at a unique point, hence $\partial F_c$ is convex.
Second, if $c\cdot A$ has negative eigenvalues, then $\partial F_c$ is empty because $F$ stretches to infinity in the directions along $-c$ and there is no corresponding supporting hyperplane in this case. Finally, when $c\cdot A$ is positive semi-definite and singular, $\partial F_c$ may consists of more than one point and hence can be non-convex provided that a few extra conditions are satisfied.

\begin{proposition}[Sufficient condition for non-convexity of $\partial F_c$]\label{th:noconv_cert}
If for $m \geq 3$, $n\ge 2$, matrix $c\cdot A$ is singular and positive semi-definite $c\cdot A \succeq 0$, $\dim(\Ker (c\cdot A))=1$, the equation $(c\cdot A)x_b=-c\cdot b$ has a solution, and for some $x_0\in\Ker (c\cdot A)$
vectors $v_k=(x_b^* A_k+b_k^*)x_0$ and $u_k=x_0^*\, A_k\, x_0$ are not collinear, then 
\begin{eqnarray}
\label{boundary}
\partial F_{c}=\{f(x_b+x_0)\,|\, x_0\in \Ker (c\cdot A)\}
\end{eqnarray}
is non-convex \cite{PolyakGryazina17}.
\end{proposition}
For the solution $x_b$ to exist, $\Ker(c\cdot A)$ has to be orthogonal to $c\cdot b$ which means that for each $x_0\in\Ker (c\cdot A)$, orthogonality condition must be satisfied $x_0(c\cdot b)^*=0$.
Then $\partial F_{c}$ is an image of one-dimensional space $x_b + t x_0$, 
\begin{eqnarray}
\partial F_{c}=f(x_b+ t x_0)=y_0+2\Re(v\,t)+u |t|^2\ ,\quad y_0=f(x_b)\ ,
\end{eqnarray}
where $x_0$ is any non-zero vector from $\Ker(c\cdot A)$.
Here we need to distinguish the complex case, $x\in \mathbb{C}^n$ and $t\in \mathbb{C}$, and the real one, $x\in \mathbb{R}^n$ and $t\in \mathbb{R}$.
In the latter case $\partial F_{c}$ would be non-convex unless two vectors $v$ and $u$ are collinear.
In the former case there are two real vectors $\Re(v)$ and $\Im(v)$.
Accordingly, $\partial F_{c}$ is non-convex unless all three vectors $\Re(v)$, $\Im(v)$, and $u$ are collinear.
Geometrically, $\partial F_c$ is a parabola (or parabolic surface in the complex case), which is not convex, unless it degenerates into a straight line.

Let us emphasize that in our analysis above we relied on $\dim(\Ker (c\cdot A))=1$.
If $\dim(\Ker (c\cdot A))>1$, the set $\partial F_{c}$ can potentially be convex even if $u$ and $v$ are not collinear (also notice in this case there might be multiple vectors $u$ and $v$).
Obviously, this criterion of non-convexity does not apply to the homogeneous (trivial $b_k$) case, since $v=0$ and is always collinear to $u$.
In this case potential boundary non-convexities are associated with the vectors $c$ for which $\dim(\Ker (c\cdot A))\ge 2$ (see Appendix \ref{non_convex_homogeneous} for further details).
Another important point here is that rewriting a complex map $f:{\mathbb C}^n \rightarrow {\mathbb R}^m$ with $\dim(\Ker (c\cdot A))=1$ as a real map $f:{\mathbb R}^{2n} \rightarrow {\mathbb R}^m$ would double the dimension of $\dim(\Ker (c\cdot A))=2$, thus rendering the Proposition \ref{th:noconv_cert} useless. 

Vectors $c$ which satisfy the conditions of the Proposition \ref{th:noconv_cert} obviously belong to set of all vectors $c$ associated with boundary non-convexities $C_{\rm ncvx}$, but might not exhaust it.
In the numerical approaches to identify boundary non-convexities we will be looking for vectors $c$ which belong to a broader set $C_-\supseteq C_{\rm ncvx}$,
\begin{align}
\label{cminorig}
C_-=\{c\in {\mathbb R}^m| c\cdot A\succeq 0,\ \dim(\Ker (c\cdot A))\ge 1,\ \forall\,x_0\in \Ker(c\cdot A), \ x_0^*(c\cdot b)=0\}\ .
\end{align}
The set $C_-$ is "larger" than $C_{\rm ncvx}$, but since the condition $\dim(\Ker (c\cdot A))=1$ is typical for singular $c\cdot A\succeq 0$ for a general map $f$, and also in the general case $u\nparallel v$, for practical purposes $C_-$ can be often "equated" with $C_{\rm ncvx}$.


To identify boundary non-convexity one can try to sample $c\in {\mathbb R}^m$, $|c|^2=1$ randomly in a hope to find $c\in C_-$ and then confirm non-convexity by checking $\dim(\Ker (c\cdot A))=1$ and non-collinearity of $u$ and $v$.
But since for a definite map $C_{-}\subset \partial {\mathcal K}\subset {\mathbb R}^m$ is a codimension one subspace in ${\mathbb R}^m$, the probability of accidentally "hitting" $c\in C_-$ is zero.
A much more efficient way to identify boundary non-convexities is outlined below.

\begin{figure}[h]
\centerline{\includegraphics[width=5cm]{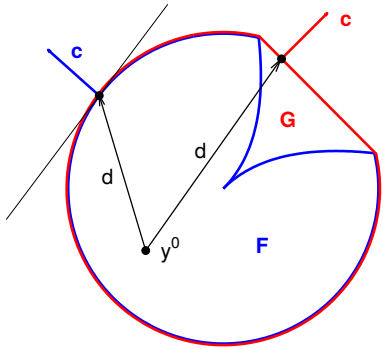}}
\caption{The idea behind identifying boundary non-convexities $c\in C_-$.}\label{fig:nonconv_idea}
\end{figure}

\subsection{Boundary oracle}
\label{sec:boundaryoracle}
The idea behind identifying boundary non-convexities is illustrated in Fig.~\ref{fig:nonconv_idea}.
Suppose we start with an internal point $y\in G$, choose a direction vector $d\in {\mathbb{R}}^m$ and identify a boundary point in that direction, $y+t d\in \partial G$, where $t$ is a numerical parameter $t\in {\mathbb R}$.
If this point happens to be a regular boundary point of $\partial F$, then locally around that point $\partial F$ and $\partial G$ coincide.
Accordingly, the supporting hyperplane which "touches" $G$ at $y+t d$ is a also a supporting hyperplane for $F$, $y+t d\in\partial F_c$ with some appropriate $c$.
In this case $\partial G_c=\partial F_c$ is convex and the corresponding $c\notin C_{\rm ncvx}$ (blue vector $c$ in Fig.~\ref{fig:nonconv_idea}).
On the contrary, if $y+t d \notin F$, since this point belongs to $G$, this implies that $F$ is not convex, $F\subsetneq G$.
We can further consider vector $c$ which is orthogonal to the supporting hyperplane to $G$ that includes $y+t d$, i.e.~$y+t d \in \partial G_c$.
Now if we consider $\partial F_c$ with the same $c$ it is {\it not} going to include $y+t d$ and will be non-convex (red vector $c$ in Fig.~\ref{fig:nonconv_idea}).

This observation provides an efficient way to identify boundary non-convexities of $F$: starting with an arbitrary point $y\in F$, randomly sample direction vectors $d$ and study the geometry near the boundary points $y+t d \in G$.

For the given $y\in G$ and $d\in {\mathbb R}^m$ the boundary point $y+t d \in \partial G$ can be efficiently obtained with help of the following Semidefinite Program (SDP) \cite{Boyd, PolyakGryazina17}
\begin{align}
\max ~t~~~~~~~& \label{BO_SPD}\\
\mathcal{H}(X) = y + t\,d,& \nonumber \\
X=X^*,\ X \succeq 0,& \nonumber\\
X_{n+1, n+1} = 1, & \nonumber
\end{align}
with variables $t\in \mathbb{R}$, $X \in \mathbb{V}^2$.
Note that this problem may not have a solution if $G$ stretches to infinity in the direction $d$.
If the solution of (\ref{BO_SPD}) satisfies $\Rg X = 1$, the corresponding boundary point of $G$ is also a boundary of $F$.
Otherwise, if $\Rg X = 1$ solution does not exist, the boundary point of the convex hull $G$ does not belong to $F$, signaling non-convexity of $F$. We note however that it is not straightforward to check  if $\Rg X=1$ solution exist as normally there are many solutions $X$ at which global optimum is achieved and standard optimization algorithms return only one of them.

The algorithm (\ref{BO_SPD}) to find boundary point of $G$ and verify if it belongs to $F$ is implemented in the accompanying library as \underline{\tt boundary\_oracle.m}.

There is also a dual formulation of the same problem which finds vector $c$, normal to the supporting hyperplane to $G$ that includes $y+t d$. It can be formulated in terms of the following SDP \cite{PolyakGryazina17}
\begin{align}
\min ~\gamma + (c\cdot y^0)& \label{dual_BO}\\
(c\cdot d) = -1 &\nonumber \\
H = \left(
\begin{array}{cc}
c\cdot A & c\cdot b \\
c\cdot b^* & \gamma \\
\end{array}
\right) &\succeq 0\ .
 \nonumber
\end{align}
This is a SDP in variables $c\in {\mathbb R}^m$ and $\gamma\in {\mathbb R}$.
As in the previous case this problem may not have a solution for certain $d$.
This algorithm is implemented in the accompanying library as \underline{\tt get\_c\_from\_d.m}.

\subsection{Non-convexity certificate}
\label{sec:nonconvex}
Equipped with boundary oracle technique (which provides both a boundary point of $G$ in a given direction as well as the normal vector $c$ at that point) we are able to discover vectors $c\in C_-$ and consequently verify if they also belong to  $C_{\rm ncvx}$.
In our approach we sample random directions $d$, obtain corresponding $c$ using \eqref{dual_BO} and check if it satisfies the conditions of the Proposition \ref{th:noconv_cert}.
This process continues unless such $c\in C_-$ is found or the number of attempts exceed some limit.
This algorithm is implemented in the accompanying library as \underline{\tt get\_c\_minus.m}.

To establish non-convexity of $F$ it is sufficient to show that $C_{\rm ncvx}$ is not empty by providing at least one non-zero $c\in C_{\rm ncvx}$.
When $b_k$ is non-trivial this can be done by using the algorithm to find $c\in C_{\rm ncvx}$ outlined above.
When $f$ is homogeneous we use a similar algorithm which identifies boundary non-convexities with $\dim(\Ker(c\cdot A))=2$, see Appendix \ref{non_convex_homogeneous}.
This algorithm is implemented in the accompanying library as \underline{\tt nonconvexity\_certificate.m}.

\begin{proposition}[Efficiency of non-convexity certificate]\label{th:prob1}
Let $d\in {\mathbb R}^m$, $|d|=1$ be a uniformly distributed on the unit sphere random variable and
$$ \varphi(d) = \left\{\begin{array}{l}
                      1, \ \ \text{ if the solution } $c$ \text{ of the problem (\ref{dual_BO}) satisfies the conditions of the Prop. \ref{th:noconv_cert}}\\
                      0, \ \ \text{ otherwise}
                    \end{array}
 \right.$$
Then for a generic map $f$ if the image $F$ is non-convex the expectation $\mathbb{E}(\varphi)>0$.
\end{proposition}

The idea of the proof is two-fold. First, we notice that for definite maps vectors $c\in C_-$ which satisfy the conditions of the Proposition \ref{th:noconv_cert} are typical in $C_-$.
Provided $F$ is non-convex, for any $y$ there is a direction $d_0$ such that $y+td_0\in G$ is not in $F$.
Moreover, because of typicality argument, vector $c$ associated with $y+td_0$ would be the one recognized by our approach as non-convex, $\varphi(d)=1$.
Second, and crucial point, any vector $d$ from a sufficiently small but finite vicinity of $d_0$ would result in the same vector $c$, as illustrated in Fig.~\ref{fig:nonconv_idea}.
Hence there is a finite probability $\mathbb{E}(\varphi)>0$ that a random vector $d$ would fall into a small but finite vicinity of $d_0$.

This proposition establishes efficiency of our stochastic non-convexity certificate.
As the number of random iterations is taken to infinity, inability of the algorithm to find $c\in C_-$ means almost surely in the probabilistic sense that the image $F$ is convex.

\section{Identifying convex subpart of $F$}
\label{sec:convexsubpart}
If a boundary non-convexity $c\in C_{\rm ncvx}$ is found, the image $F$ is non-convex.
In this case it would be desirable to identify a convex subset of $F$ which would be maximally large in size and simple to deal with.
The approach of \cite{Dymarsky_cuttingArX} is to find a particular hyperplane which would split $F$ into two parts such that the compact part is convex.
More concretely, for some $c_+ \in {\mathcal K}_+$ such that $c_+\cdot A\succ 0$, we would like to find maximal $z=z_{\rm max}$ such that the set
\begin{align}
\label{Fz}
F_z=\{y\in {\mathbb R}^m |\, y\in F,\ c\cdot y\le z\}\subset F
\end{align}
is convex.
The following proposition explains how to calculate $z_{\rm max}$.
\begin{proposition}[Convex cut]
	\label{lemma:convex_cut}
Let $c_+\in {\mathcal K}_+$ such that $A_+\equiv c_+\cdot A\succ 0$, and $x^0=-A_+^{-1}b_+$, where $b_+=c_+\cdot b$.
Then $F_z$ \eqref{Fz} with $z=z_{\rm max}$ given by
\begin{align}
\label{zdef}
z_{\rm max}=\min_{c\in C_-} \|(c\cdot A)^{-1}(c\cdot b)-x^0\|^2_+
\end{align}
is convex \cite{Dymarsky_cuttingArX}.
Here $\|x\|_+^2$ is defined as $\|x\|_+^2\equiv x^* A_+ x$.
Here and below $(c\cdot A)^{-1}$ stands for a pseudo-inverse of $(c\cdot A)$ when the latter is singular.
\end{proposition}

The geometrical logic behind \eqref{zdef} is as follows.
Each $c\in C_-$ defines a potentially non-convex boundary region \eqref{boundary}, which is called "flat edge" in \cite{Dymarsky_cuttingArX}.
We consider the projection of this region on $c_+$ and immediately find that for $c\in C_-$,
\begin{align}
z(c)=\min_{y\in \partial F_c} (c_+\cdot y)=\|(c\cdot A)^{-1}(c\cdot b)-x^0\|^2_+\ .
\end{align}
This simply means that the "flat edge" (potentially non-convex boundary) $\partial F_c$ does not stretch "beyond" the hyperplane $c_+\cdot y=z(c)$, i.e.~all points $y\in \partial F_c$ satisfy $c_+\cdot y\ge z(c)$.
The value of \eqref{zdef} defined as $z_{\rm max}=\min_{c\in C_-} z(c)$ guarantees that no boundary non-convexity stretches beyond $c_+\cdot y=z_{\rm max}$.
This is clearly a necessary condition for $F_{z_{\rm max}}$ to be convex.
Moreover, it is also sufficient \cite{Dymarsky_cuttingArX}.

Below we formulate the algorithm to find $z_{\rm max}$ numerically by calculating
\begin{equation}
\label{eq:z_max} z_{\max} = \min\limits_{c\in C_{-}} z(c)\ ,
\end{equation}
using gradient descent along $C_-$.
To simplify the following presentation we perform a linear change of variables $x\rightarrow x-x^0$, accompanied by $b_i\rightarrow b_i+A_i x^0$ and the shift $y_i \rightarrow y_i-(x^0)^* A_i x^0 -(x^0)^* b_i-b_i^* x^0$.
In the new coordinates quadratic map still has the conventional form \eqref{real} or \eqref{complex}.
Next we perform a linear transformation $x\rightarrow \Lambda x$ where $A_+=\Lambda^* \Lambda$.
In the new coordinates $A_+={\mathbb I}$ is the identity matrix and $\|x\|^2=x^*x$ is the regular Euclidean norm.
New $b_i$ also satisfies $c_+\cdot b=0$. In the new coordinates we introduce
\begin{eqnarray}
\label{v}
v(c)&=&(c\cdot A)^{-1} (c\cdot b),\,\mbox{and then} \\
z(c)&=&v^*v\ . \label{zv}
\end{eqnarray}
Notice that even though $c\cdot A$ is singular for $c\in C_-$, $v(c)$  satisfies $(c\cdot A)v(c)=c\cdot b$.

\subsection{Geometry of $C_-$}
\label{section:geometry}
To implement gradient descent along $C_-$ we would like first to understand its dimensionality.
First we notice that $C_-\subset \partial {\mathcal K}$.
The boundary $\partial {\mathcal K}$ can be parametrized by all vectors $c$ such that $c\cdot c_+=0$.
Indeed for any vector $c$, vector $p(c)$
\begin{equation}
p_i\equiv c_i-(c_+)_i \lambda_{\rm min}(c\cdot A)
\end{equation}
belongs to $\partial {\mathcal K}$ as the associated matrix $p\cdot A\succeq 0$ and singular.
Here $\lambda_{\rm min}(c\cdot A)$ stands for the smallest eigenvalue of $c\cdot A$.
Because of $c_+\cdot b=0$ we also have $c\cdot b=p\cdot b$.
Furthermore, function $z(c)$ is invariant under rescaling of $c$: $z(c)=z(\mu c)$ for any $\mu>0$.
Hence for the purpose of finding $z_{\rm max}$ numerically we can redefine $C_-$ as follows
\begin{equation}
C_-=\{ c\in {\mathbb R}^m|\, c\cdot c_+=0,\ |c|^2=1,\ \dim\Ker(p\cdot A)= 1,\ \forall\, x_0\in \Ker(p\cdot A),\ x_0^*(c\cdot b)=0 \}\subset {\mathbb S}^{m-2}\ .
\end{equation}
As in the case of section \ref{boundary_nonconvexity} and in the same sense $C_-$ is approximately equal to $C_{\rm ncvx}$. Although $C_-$ includes vectors $c$ for which $\dim\Ker(p\cdot A)>1$, these vectors have measure zero inside $C_-$, hence this condition does not reduce dimensionality of $C_-$.
The important condition is $x_0^*(c\cdot b)=0$, which imposes a real-valued or complex-valued constraint, reducing the dimension of $C_-$ by one or by two correspondingly.
Hence we conclude that $C_-$ is an $(m-3)$-dimensional subset in ${\mathbb S}^{m-2}$ in the real case, and $(m-4)$-dimensional subset in ${\mathbb S}^{m-2}$ in the complex case.

When $m=4$ and $\mathbb{V}=\mathbb{C}$, the set $C_-$ consists of discrete points inside ${\mathbb S}^{2}$.
In that case all $c_-\in C_-$ can be found analytically or numerically using \underline{\tt get\_c\_minus.m}, and $z_{\rm max}$ can be calculated explicitly.
An example of such an analytic calculation -- for the solvability set of Power Flow equation for a 3-bus system -- can be found in \cite{DymarskyTuritsyn}.
Similar logic applies for real quadratic maps with $m=3$.
An example when all $c_-\in C_-$ are calculated both analytically and numerically is presented in the Section \ref{num_res}.

\subsection{Continuous case}
\label{contcase}
In the general case, when $m> 3$ and $m>4$ for the real and complex maps correspondingly, the set $C_-$ will be a continuous subset within ${\mathbb S}^{m-2}$ of codimension one or two.
A priori it may consists of several disjoint patches and have self-intersections.
We will assume that $C_-$ is smooth modulo special points of measure zero. 
Once a point $c_-\in C_-$ is identified, we would like to perform a gradient descent along $C_-$ to minimize $z(c)$.
This process should repeat for all patches of $C_-$.
In practice we will repeat \underline{\tt get\_c\_minus.m} and for each found $c_-\in C_-$ perform a gradient descent, keeping the smallest value of $z(c)$ among all iterations.

Let us now assume that $c(t):{\mathbb R}\rightarrow C_-$ is a smooth trajectory of gradient descent inside $C_-$ (here we hypothetically take the step of gradient descent to be infinitesimally small).
Then for each $t$ it must satisfy $|c|^2=1$, $c\cdot c_+=0$ and $x_0^*(c\cdot b)=0$.
By differentiating these conditions with respect to $t$ we find the following set of linear constraints on $\dot{c}$ (see Appendix \ref{sec:gradient} for derivation):
\begin{eqnarray}
\label{const}
\dot{c}\cdot c(t)=0\ ,\quad \dot{c}\cdot c_+=0\ ,\quad \dot{c}\cdot n(c)=0\ ,\\
n_i=x_0^*q_i\ ,\quad q_i=b_i -(A_i-(x_0^* A_i x_0){\mathbb I})Q(c)^{-1}(c\cdot b)\ ,\quad Q(c)=p(c)\cdot A\ .
\label{q}
\end{eqnarray}
Here $x_0$ is a normalized vector $|x_0|^2=1$, $x_0\in \Ker (p\cdot A)$.
If $m=4$ in the real case or $m=5$ in the complex case constraints \eqref{const} uniquely specify the direction of possible gradient descent $\dot{c}$ up to an overall sign.
But when $m$ is larger the direction of the gradient descent follows from \eqref{zv} (see Appendix \ref{sec:gradient}):
\begin{equation}
	\label{eq:gradient}
	({\nabla} z)_i=\frac{\partial z}{\partial c_i}=2\Re(v^*Q^{-1}q_i)\ .
\end{equation}
This expression automatically satisfies
\begin{equation}
\nabla z\cdot c(t)=0\ ,\quad \nabla z\cdot c_+=0\ ,
\end{equation}
but $\nabla z\cdot n\neq 0$.
To impose $\dot{c}\cdot n=0$, we introduce a projector $P(\nabla z)$. In the real case it has the form
\begin{equation}
P[{\nabla} z]=\vec{\nabla} z -\vec{n} (n\cdot {\nabla} z)/|n|^2\ .
\end{equation}
In the complex case there are two vectors $n_1=\Re(n)$ and $n_2=\Im(n)$ and therefore
\begin{align}
P[{\nabla} z]=\vec{\nabla} z -\vec{n}_1 a - \vec{n}_2 b\ ,\qquad \qquad\qquad \qquad\qquad \\
a={(n_1\cdot {\nabla} z)|n_2|^2-(n_2\cdot {\nabla} z)(n_1\cdot n_2)\over |n_1|^2 |n_2|^2-(n_1\cdot n_2)^2},\
b={(n_2\cdot {\nabla} z)|n_1|^2-(n_1\cdot {\nabla} z)(n_1\cdot n_2)\over |n_1|^2 |n_2|^2-(n_1\cdot n_2)^2}\ .
\end{align}

\begin{figure}[h]
\centerline{\includegraphics[width=0.5\textwidth]{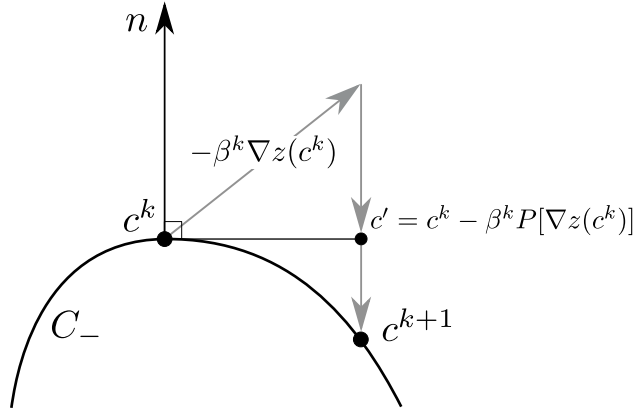}}
\caption{The gradient projection method}
\label{fig:cminus_grad}
\end{figure}

Applying the projector ensures that $\dot{c}$ changes along $C_-$ provided the step of the gradient descent is infinitesimally small.
In the numerical implementation this is clearly not the case.
Hence the full algorithm will consist of iteratively applying two steps: the step of gradient descent along the tangential direction to $C_-$ and then an additional projection $\pi_{C_-}$ onto $C_-$.
The initial value $c^{(1)}$ is provided by the call of \underline{\tt get\_c\_minus.m}.
Assuming at step $k\ge 1$ vector $c=c^{(k)}$, the iteration is as follows (see Figure \ref{fig:cminus_grad})
\begin{equation}
c^{(k+1)}=\underbrace{\pi_{C_-}}_{\mathrm{projector}}(\underbrace{c^{(k)}-\beta^k P[\nabla z(c^{(k)})]}_{\mathrm{gradient \,\,step}})\ .
\end{equation}
Here $\beta^k$ is the length of the gradient descent step at iteration $k$ and the project $\pi_{C_-}$ has to be defined separately for real and complex cases.



\noindent {\bf Projector in $\mathbb{V}=\mathbb{R}$ case}.
After calculating $c'=c^{(k)}-\beta^k P[\nabla z]$ this vector would automatically satisfy $c'\cdot c_+=0$ but since $\beta^k$ is finite, it does not necessarily belongs to $C_-$.
To project the result onto $C_-$ we will consider a family $\tilde{c}(\lambda)=c'+\lambda \vec{n}(c^{(k)})$ and will find $\lambda$ such that $\tilde{c}(\lambda)\in C_-$.
To that end we define function $m$ as the "distance" to $C_-$ in terms of the following dot product,
\begin{eqnarray}
m(\lambda)=x_0^*\big(\tilde{c}(\lambda)\big)\big(\tilde{c}(\lambda)\cdot b\big),
\end{eqnarray}
where $x_0(\tilde{c}) \in \Ker \big(p(\tilde{c})\cdot A\big)$, $|x_0|^2=1$, and the overall sign of $x_0$ is chosen such that the dot product $x_0^*\big(\tilde{c}(\lambda)\big) x_0(c')\ge 0$. The latter condition is necessary to make the function $m$ continuous.
Then we try to find $\lambda$ such that $m(\lambda)=0$ which is equivalent to $\tilde{c}(\lambda)/|\tilde{c}|\in C_-$.
Function $m(\lambda)$ is continuous in the vicinity of $\lambda=0$ provided $\Rg Q(\lambda=0)=n-1$.
We find the root of $m(\lambda)$ numerically using the bisection method on the interval $\lambda\in [-\lambda_0,\lambda_0]$ with $\lambda_0=\|c-c'\|$ as a heuristic estimate for the maximal possible value of $\lambda$.

If for some $\lambda$, $m(\lambda)=0$, the projection step was a success, and the new point $c^{(k+1)}=\tilde{c}(\lambda)/|\tilde{c}(\lambda)|\in C_-$.
If function $m(\lambda)$ does not change sign on the interval $[-\lambda_0,\lambda_0]$, or at some $\lambda$, $\Rg Q(\tilde{c}(\lambda))\neq n-1$, we reduce the gradient step $\beta^k$, recalculate $c'$ and try the projection again.

The gradient steps continue until the gradient $\nabla z$ becomes collinear with $n$, which signals that a local minimum of $z(c)$ is reached, or $\dim\Ker(Q(c))>1$ which means the gradient descent trajectory reached a boundary point of $C_-$.

\noindent{\bf Projector in $\mathbb{V}=\mathbb{C}$ case}.
First step is the same: we define $c'=c^{(k)}-\beta^k P[\nabla z]$. Since in the complex case there are two normal vectors, $\vec{n}_1$ and $\vec{n}_2$, the projection procedure is different.
We define a "distance" to $C_-$ in terms of the $c$-dependent "norm-square" function
\begin{equation}
\rho(c)=w^* w\ ,\quad w(c)=x_0(c)^*(c\cdot b)\ .
\end{equation}
Obviously $\rho$ is positive semi-definite and $\rho(c)=0$ if and only if $c\in C_-$.
It is also a continuous function of $c$ provided $\dim\Ker(Q(c))=1$. To find $\tilde{c}$ such that $\rho(\tilde{c})=0$ we apply gradient descent starting from $c'$
and using (see Appendix \ref{sec:gradient} for derivation)
\begin{eqnarray}
\frac{\partial \rho(c)}{\partial c_i}=2\Re(w^*(x_0^*q_i))\ .
\end{eqnarray}

We note that the projector in the complex case can be also used in the real case. However, the binary search is substantially faster than the gradient descent, resulting in a speedup for real maps.


\section{Examples}
\label{num_res}
In this section we test the proposed algorithms on a range of several multidimensional maps.
Some of the maps are artificially or randomly generated, while others describe Power Flow equations for certain energy networks. Our main focus is to identify convex subpart of $F$ as described in section \ref{sec:convexsubpart}. This will automatically include 
certifying (non)-convexity using the algorithm of section \ref{sec:nonconvex} and
finding boundary non-convexities using boundary oracle of section \ref{sec:boundaryoracle}. Each test below consists of two parts: numerical (applying {\tt get\_z\_max}) and analytical (if an analytic analysis is possible). 

All examples discussed in this section are implemented as test cases in the library CAQM. Running corresponding {\tt .m} files will generate the data presented in this section (although the algorithms are stochastic in nature the random seed remains the same hence re-running the program will lead to the identical results).


\paragraph{Example 1. Artificial ${\mathbb R}^3\rightarrow {\mathbb R}^3$ map.}
{\em See file \underline{\tt examples/article\_example01.m}.}

\noindent We start with a $\mathbb{R}^3 \rightarrow \mathbb{R}^3$ quadratic mapping specified by
\begin{align*}\arraycolsep=1.4pt\def\arraystretch{1}
A_1 & = \left(\begin{array}{ccc} 1 & 1 & 1 \\ 1 & 2 & 0 \\ 1 & 0 & 2 \\ \end{array}\right), &
A_2 & = \left(\begin{array}{ccc} 3 & -1 & 0 \\ -1 & 0 & -1 \\ 0 & -1 & 1 \\ \end{array}\right), &
A_3 & = {\mathbb I}\ , \\
b_1 & = \left(\begin{array}{ccc} 1 & 1 & 1 \\ \end{array}\right)^T, &
b_2 & = \left(\begin{array}{ccc} 1 & 0 & -1 \\ \end{array}\right)^T, &
b_3 & = \left(\begin{array}{ccc} 0 & 0 & 0 \\ \end{array} \right)^T.
\end{align*}
It is clear that $A_3$ is positive-definite, hence this map is definite.
We choose $c_+=(0,0,1)^T$. Next, we analytically look for vectors $c\in C_-$ defined in \eqref{cminorig}.
We appropriately parametrize vector $c\in C_-$ first and solve an algebraic constraint $x_0^*(c\cdot b)=0$. The resulting $x_0$ is used to find $c$ from the relation $(c\cdot A)x_0=0$. This is done in the accompanying Mathematica notebook {\tt article\_example01.nb}. As a last step solutions for which $(c\cdot A)$ is not semi-definite should be tossed out.  Eventually, 
one finds three distinct vectors  $c\in C_-$ (up to an overall rescaling),
\begin{equation}
\label{threecm}
c_1^T=\left(1, 0, 0 \right),\quad c_2^T = \left(2, 1, -2 \right),\quad c_3^T = \left(0.7227, 3.4347, 1 \right).
\end{equation}
Corresponding values of $z(c)$ \eqref{zdef} are as follows
\begin{equation}
z(c_1)={1\over 3}\ ,\quad z(c_2)={74\over 75}\ ,\quad z(c_3)=0.3656\ .
\end{equation}
Three different vectors \eqref{threecm} means there are three boundary non-convexities as illustrated in Fig. \ref{fig:art_2D}.
There we plot two different 2D sections of the image $F$ corresponding to $y_3=1/3$ and $y_3=4$.
In the first case $y_3=1/3$ and the section is convex, but not strongly convex at the point  highlighted in the Figure. $y_3=1/3$ is the critical value at which the boundary non-convexity associated with $c_1$ develops.
In the second case $y_3=4$ all three boundary non-convexities are clearly visible, together with the corresponding points of "flat edge" $\partial F_c$.

Running the algorithm numerically identifies all three boundary non-convexities and yields the correct value $z_{\rm max}=1/3$.

\begin{figure}[h]
\includegraphics[width=8cm]{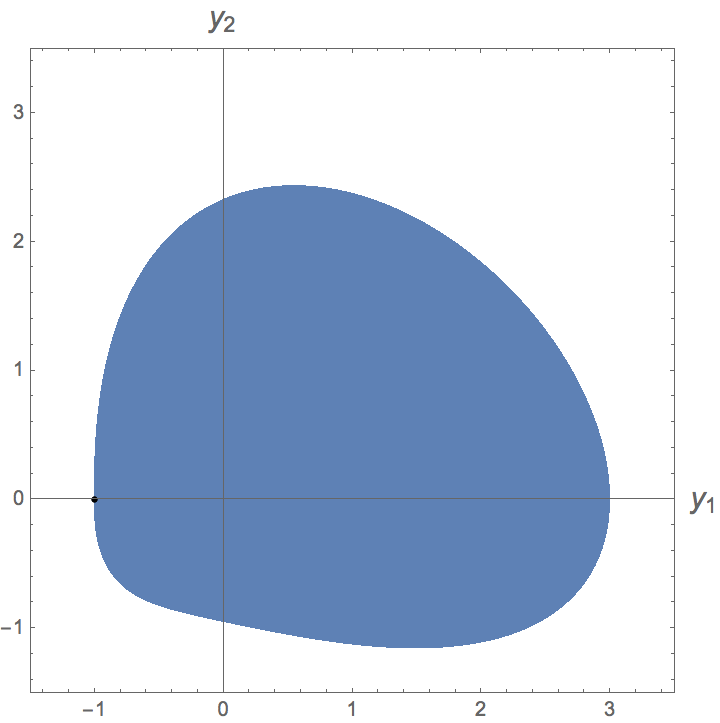}
\includegraphics[width=8cm]{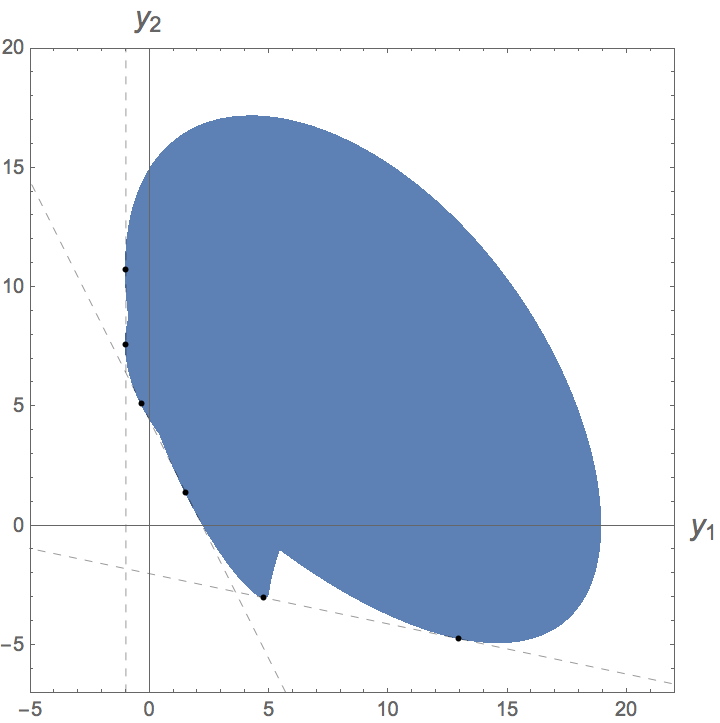}
\caption{Two sections of the feasibility domain: first we fix $y_3 = 1/3$ and obtain convex section, then for $y_3 = 4$ the section is non-convex.}
\label{fig:art_2D}
\end{figure}

\paragraph{Example 2. Power Flow system of \cite{Ortega15}.}
{\em See file \underline{\tt examples/article\_example02.m}.}\\
\noindent This example of quadratic map is from the article \cite{Ortega15}. It describes a 3-bus Power System with constant power loads. In mathematical terms the problem considred there is the feasibility problem of section \ref{certif} for the $\mathbb{R}^3 \rightarrow \mathbb{R}^3$ quadratic map
\begin{eqnarray}
\label{ex2}
P_1(x) & =& x_1^2 - 0.5 x_1 x_2 + x_1 x_3 - 1.5 x_1 \\ \nonumber
P_2(x) & =& x_2^2 - 0.5 x_1 x_2 - x_2 x_3 + 0.5 x_2 \\ \nonumber
P_3(x) & =& x_3^2 - 2 \epsilon x_3( x_1 + x_2) - x_3, \quad \epsilon = 0.01.
\end{eqnarray}
Here we investigate convexity of the map \eqref{ex2}.
For $c_+ = \left( 2,2,1 \right)^T/3$ using the approach discussed in the case of Example 1, we analytically obtain a unique $c = (0.3169, 0.9196, 0.2322)^T\in C_-$ associated with a boundary non-convexity. Further details can be found in Mathematica notebook {\tt article\_example02.nb}.
Running {\tt get\_z\_max} identifies this unique boundary non-convexity and finds $z_{\rm max}=0.0283$.

\paragraph{Example 3. AC Power Flow system of \cite{DymarskyTuritsyn}.}
{\em See file \underline{\tt examples/article\_example03.m}.}\\
\noindent We consider a tree unbalanced 3-bus AC Power Flow system (1 slack, 2 PQ-buses) described by the admittance matrix
$$Y = \left(\arraycolsep=1.4pt\def\arraystretch{1}\begin{array}{ccc} -1-i & 1+i & 0 \\ 1+i & -2-i & 1 \\ 0 & 1 & -1 \\ \end{array} \right).$$ 
The feasibility region in the space of $y = \left( P_2, Q_2, P_3, Q_3 \right)^T$, where $P_i$ and $Q_i$ denote active and reactive power on the $i$-th bus, is an image $F$ of a $\Cc^2\rightarrow \R^4$ quadratic map associated with the corresponding Power Flow equations. 
A complete analytic analysis of the feasibility region was performed in \cite{DymarskyTuritsyn} (details can be also found in the Mathematica notebook {\tt article\_example03.nb}), where it was shown that $F$ is non-convex with a unique vector $c=(0,0,-1,-1)^T/\sqrt{2}\in C_-$ and $z_{\rm max}=1/\sqrt{2}$ for $c_+=(1,1,0,0)^T/\sqrt{2}$. 
Running this example numerically yields the same result.

\paragraph{Example 4. AC Power Flow system of \cite{Suetin15}.}
{\em See file \underline{\tt examples/article\_example04.m}.}\\
This is an example of 3-bus AC Power Flow network with a slack, PV and PQ-buses from \cite{Suetin15}, see Fig.~\ref{fig:3bus}.
Besides a more involved structure of the power network in comparison with the Example 3, another important difference is that the entries of the corresponding admittance matrix are not integers.
Hence the corresponding quadratic map is free of accidental degeneracies.

\begin{figure}[htb]
\centering \includegraphics[width=9cm]{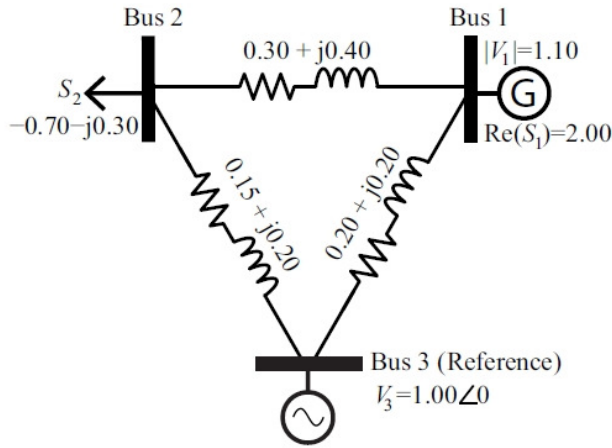}
\caption{Three-bus system}
\label{fig:3bus}
\end{figure}
The Power flow equations are as follows, 
{\tiny \begin{align*}
\arraycolsep=1.4pt\def\arraystretch{1}
P_1 & = x^T \left(
              \begin{array}{cccc}
                3.7 & -0.6 & 0 & -0.8 \\
                -0.6 & 0 & 0.8 & 0 \\
                0 & 0.8 & 3.7 & -0.6 \\
                -0.8 & 0 & -0.6 & 0 \\
              \end{array}
            \right) x
             + 2 \left(\begin{array}{c} -1.25\\0\\1.25\\0\end{array}\right) x,\\
U_1 & = x_1^2 + x_3^2, \\
P_2 & =x^T \left(
              \begin{array}{cccc}
              0 & -0.6 & 0 & 0.8 \\
                -0.6 & 3.6 & -0.8 & 0 \\
                0 & -0.8 & 0 & -0.6 \\
                0.8 & 0 & -0.6 & 3.6 \\
              \end{array}
            \right) x
             + 2 \left(\begin{array}{c} 0\\-1.2\\0\\1.6\end{array}\right) x,\\
Q_2 & =x^T \left(
              \begin{array}{cccc}
            0 & -0.8 & 0 & -0.6 \\
                -0.8 & 4.8 & 0.6 & 0 \\
                0 & 0.6 & 0 & -0.8 \\
                -0.6 & 0 & -0.8 & 4.8 \\
              \end{array}
            \right) x + 2 \left(\begin{array}{c} 0\\-1.6\\0\\-1.2\end{array}\right) x.
\end{align*}
}
We define $x = (\text{Re} V_1,~ \text{Re} V_2,~ \text{Im} V_1,~ \text{Im} V_2)^T$ and $V_3 = 1$ for the slack bus.
Converting notations to the conventional form (\ref{complex}) one finds:
\begin{align*}
A'_1 & = \left(\arraycolsep=1.4pt\def\arraystretch{1}\begin{array}{cc} 3.7  & -0.6+0.8i\\ -0.6-0.8i & 0 \end{array}\right), &
A'_2 & = \left(\arraycolsep=1.4pt\def\arraystretch{1}\begin{array}{cc} 1 & 0\\ 0 & 0\\ \end{array}\right),\\
A'_3 & = \left(\arraycolsep=1.4pt\def\arraystretch{1}\begin{array}{cc} 0 & -0.6-0.8i\\ -0.6+0.8i & 3.6 \end{array}\right), &
A'_4 & = \left(\arraycolsep=1.4pt\def\arraystretch{1}\begin{array}{cc} 0 & -0.8+0.6i\\ -0.8-0.6i & 4.8 \end{array}\right)\\
b'_1 & = \left(\arraycolsep=1.4pt\def\arraystretch{1}\begin{array}{c} -1.25+1.25i\\ 0 \end{array}\right), &
b'_2 & = \left(\arraycolsep=1.4pt\def\arraystretch{1}\begin{array}{c} 0\\ 0 \end{array}\right),\\
b'_3 & = \left(\arraycolsep=1.4pt\def\arraystretch{1}\begin{array}{c} 0\\ -1.2+1.6i \end{array}\right), &
b'_4 & =\left(\arraycolsep=1.4pt\def\arraystretch{1}\begin{array}{c} 0\\ -1.6-1.2i \end{array}\right)
\end{align*}
Mathematically, this is a $\Cc^2\rightarrow \R^4$ map. We choose vector 
\begin{equation}
c_+=(0.7991, -0.3533, 0.3924, 0.2876).
\end{equation}
Analytically we find two vectors $c\in C_-$ associated with boundary non-convexity: $c=(0,1,0,0)$  and $c=(337/328, -27971/6560, 1, -321/328)$ (see accompanying Mathematica notebook {\tt article\_example04.nb}). First vector yields $z(c)=1.4512$, while second vector gives a much larger value. 
Starting with an initial guess $z^{\rm guess}_{\max}=1.7901$ the numerical algorithm returns $z_{\rm max}=1.4506$. The discrepancy in the third digit is due to numerical precision in the function {\tt is\_nonconvex}.


\paragraph{Examples 5. Artificial $\R^4\to \R^4$ map.}

{\em See file \underline{\tt examples/article\_example05.m}.}

\noindent We consider an $\R^4\to \R^4$ map
\begin{center}
$A_1={\mathbb I}$,
$A_2=\left(\begin{array}{cccc} 1 & 0 & 1 & 0 \\ 0 & 2 & -1 & 4 \\ 1 & -1 & 0 & 0 \\ 0 & 4 & 0 & 0 \\ \end{array}\right)$,
$A_3=\left(\begin{array}{cccc} 0 & 0 & 0 & -1 \\ 0 & 3 & -1 & 0\\ 0 & -1 & -1 & 0\\ -1 & 0 & 0 & -1 \end{array}\right)$,
$A_4=\left(\begin{array}{cccc} 4 & 0 & 1 & 2\\ 0 & 0 & 0 & 4\\ 1 & 0 & 0 & 0\\ 2 & 4 & 0 & -2 \end{array}\right)$;

$b_1=\left(\begin{array}{c} 0\\ 0\\ 0\\ 0 \end{array}\right)$,
$b_2=\left(\begin{array}{c} 1\\ 0\\ 0\\ 0 \end{array}\right)$,
$b_3=\left(\begin{array}{c} 1\\ 1\\ 0\\ 0 \end{array}\right)$,
$b_4=\left(\begin{array}{c} 0\\ 0\\ 0\\ 1 \end{array}\right)$;
$c_+=\left(\begin{array}{c} 1\\ 0\\ 0\\ 0 \end{array}\right)$.
\end{center}

\begin{figure}[h]
	\includegraphics[width=9cm]{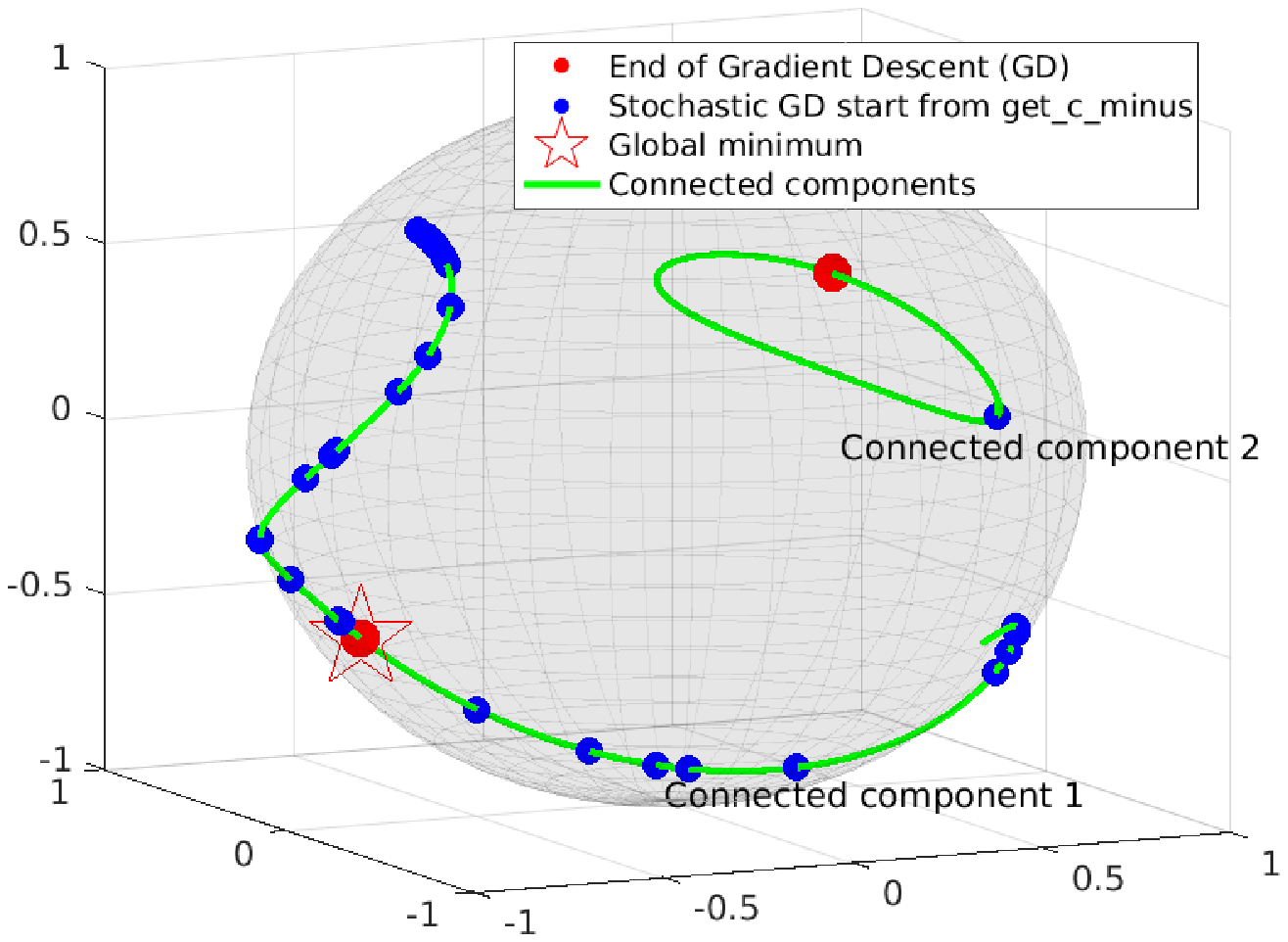}
	\includegraphics[width=9cm]{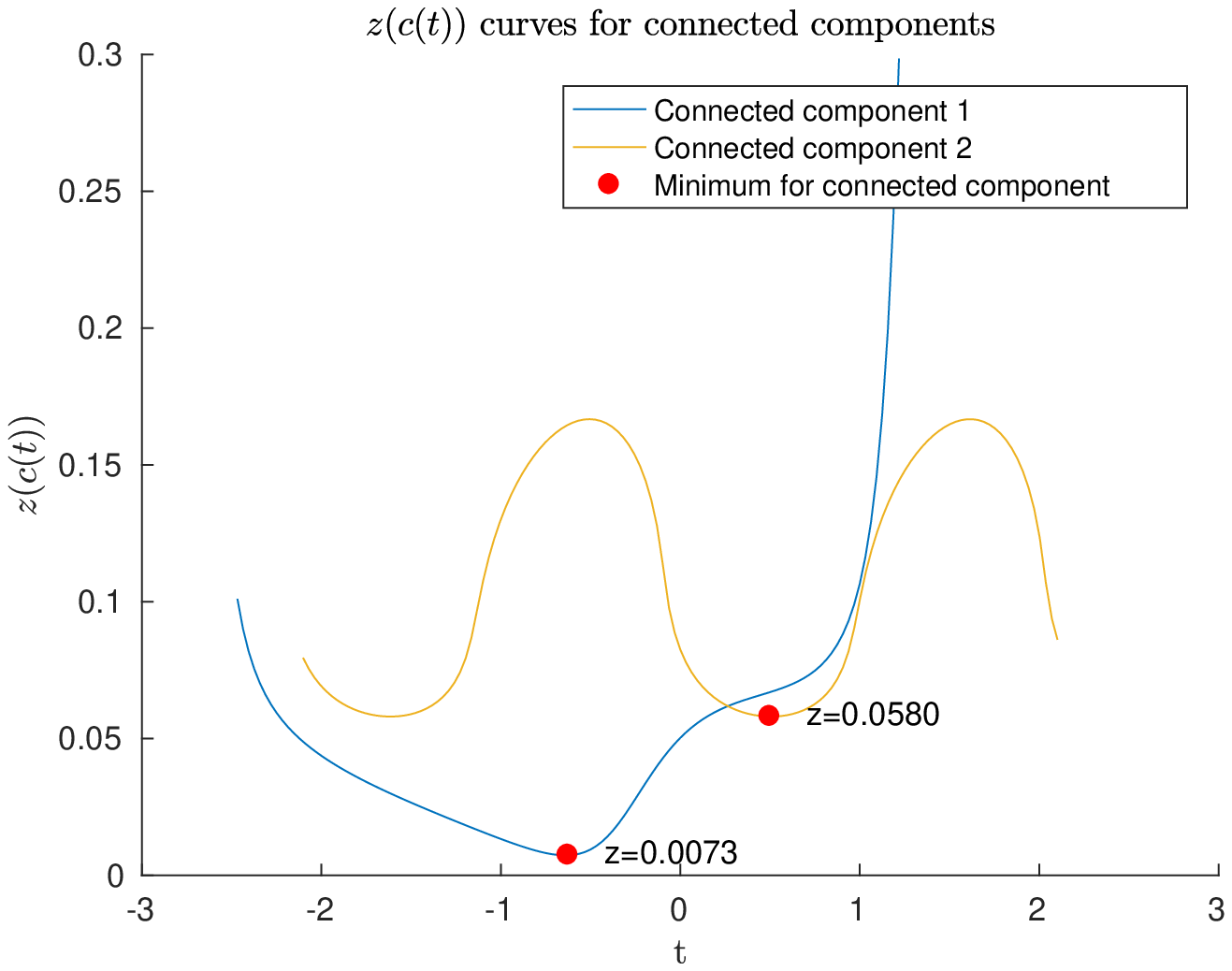}
\caption{Gradient descent along $C_-$ (left) and values of $z(c)$ (right) for the  $\R^4\to\R^4$  map of Example 5.
{\em Code for generating these figures is in the CAQM repository at {\tt\em examples/figures/article}}.}\label{fig:R4R4_05_c_minus}
\end{figure}
In the case of $\R^4\to \R^4$ map the set $C_-\subset {\mathbb S}^2$ is one dimensional, see section \ref{section:geometry}.
In this particular case it consists of at least two connected components. 
Therefore, this example tests the gradient descent method described in section \ref{contcase}. The results of a particular run are shown in Figure \ref{fig:R4R4_05_c_minus} (left) as a projection onto ${\mathbb S}^2$ orthogonal to $c_+$.
The numerical algorithm discovers plenty of starting points $c\in C_-$  for the gradient descent (shown in blue in Figure \ref{fig:R4R4_05_c_minus}). 
End points of the gradient descent for each connected component are colored red, and the global minimum is marked with a star.
These numerical results are compared with the semi-analytic results obtained as follows. When $C_-$ is one-dimensional the direction of the gradient $\nabla z(c)$ must be aligned with $\dot{c}\in \left(\mbox{Lin}\{c,c_+,n\}\right)^{\bot}$. Using the latter expression we numerically constructed $C_-$ and found it to be in agreement with the one obtained by the full algorithm.
Connected components obtained using this method are shown in green in Figure~\ref{fig:R4R4_05_c_minus} (left).
Finally, in Figure~\ref{fig:R4R4_05_c_minus} (right) we show a plot of $z(c(t))$ as a function of a parameter $t$ along $C_-$. This plot confirms that our algorithm correctly identifies the direction of the descent and chooses a global minimum of $z(c)$.

One of the components of $C_-$ has a topology of a ring, and another one is an open interval with end points satisfying $\Rg Q(c)=n-2$.
Our gradient descent algorithm terminates once the point $\Rg Q(c)=n-2$ is encountered. Numerically in this case the algorithm finds $z_{\rm max}=0.007325$.


\paragraph{Examples 6. Randomly generated $\R^4\to \R^4$ map.}
{\em See file \underline{\tt examples/article\_example06.m}.}
The map of the Example 5 was artificially constructed (with all coefficients being integers), which simplifies analytic analysis. 
Example 6 considers a randomly generated map 
\begin{center}
	$A_1=\left(\arraycolsep=1.4pt\def\arraystretch{1}
	\begin{array}{cccc}
	3.6434 & 1.1990 & 1.2652 & 0.7187\\
	1.1990 & 2.7936 & 1.0245 & 1.4263\\
	1.2652 & 1.0245 & 3.5808 & 1.3879\\
	0.7187 & 1.4263 & 1.3879 & 3.6670\\
	\end{array}\right)$,
	$A_2=\left(\arraycolsep=1.4pt\def\arraystretch{1}
	\begin{array}{cccc}
	1.0288 & 1.0841 & 1.3780 & 0.2665\\
	1.0841 & 0.8139 & 1.0672 & 0.9619\\
	1.3780 & 1.0672 & 1.0681 & 0.7686\\
	0.2665 & 0.9619 & 0.7686 & 0.9904\\
	\end{array}\right)$,\\
	$A_3=\left(\arraycolsep=1.4pt\def\arraystretch{1}
	\begin{array}{cccc}
	1.7014 & 1.1434 & 0.9301 & 1.2243\\
	1.1434 & 1.6308 & 1.7445 & 1.4684\\
	0.9301 & 1.7445 & 1.2251 & 1.7913\\
	1.2243 & 1.4684 & 1.7913 & 0.4557\\
	\end{array}\right)$,
	$A_4=\left(\arraycolsep=1.4pt\def\arraystretch{1}
	\begin{array}{cccc}
	1.4773 & 0.6695 & 1.1375 & 1.5947\\
	0.6695 & 1.2519 & 1.6432 & 1.3098\\
	1.1375 & 1.6432 & 1.5381 & 0.5984\\
	1.5947 & 1.3098 & 0.5984 & 0.2417\\
	\end{array}\right)$,\\
	$b_1=\left(\arraycolsep=1.4pt\def\arraystretch{1}
	\begin{array}{c}
	0.7689\\
	0.1673\\
	0.8620\\
	0.9899\\
	\end{array}\right)$,
	$b_2=\left(\arraycolsep=1.4pt\def\arraystretch{1}
	\begin{array}{c}
	0.1897\\
	0.4950\\
	0.1476\\
	0.0550\\
	\end{array}\right)$,
	$b_3=\left(\arraycolsep=1.4pt\def\arraystretch{1}
	\begin{array}{c}
	0.4981\\
	0.9009\\
	0.5747\\
	0.8452\\
	\end{array}\right)$,
	$b_4=\left(\arraycolsep=1.4pt\def\arraystretch{1}
	\begin{array}{c}
	0.8627\\
	0.4843\\
	0.8449\\
	0.2094\\
	\end{array}\right)$. \end{center}

\begin{figure}[H]
	\includegraphics[width=9cm]{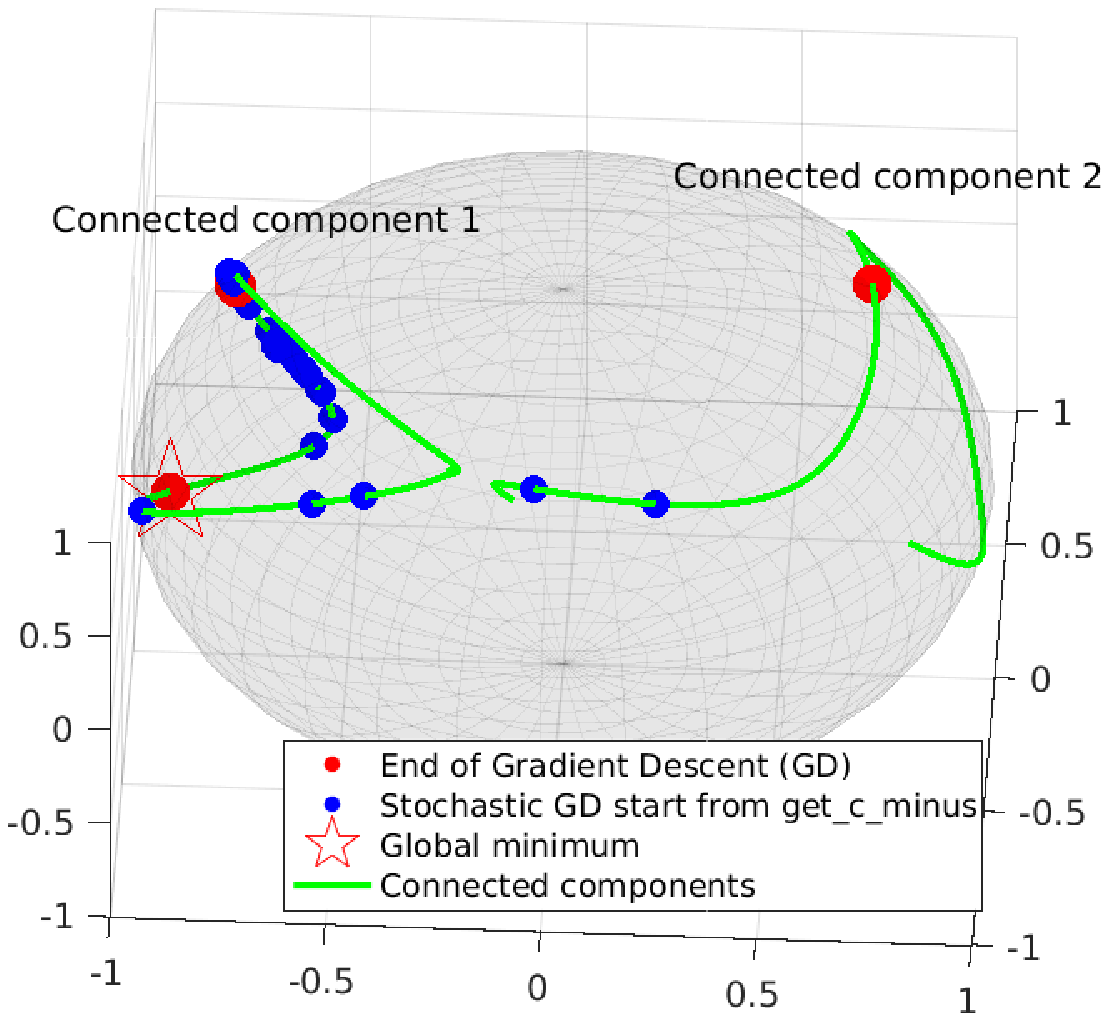}
	\includegraphics[width=9cm]{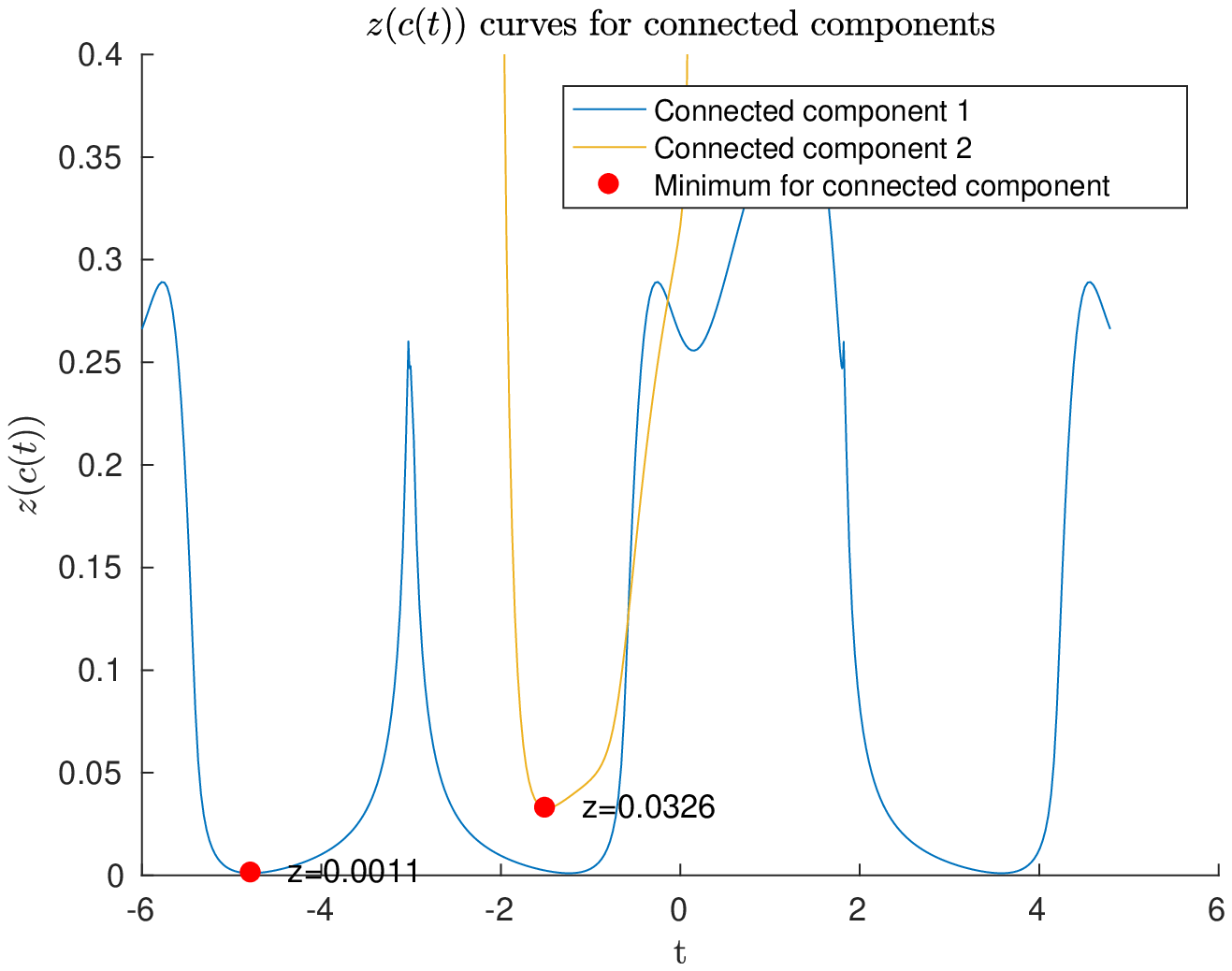}
\caption{Gradient descent along $C_-$ (left) and values of $z(c)$ (right) for the  $\R^4\to\R^4$  map of Example 6.
{\em Code for generating these figures is in the CAQM repository at {\tt\em examples/figures/article}}.}\label{fig:R4R4_06_c_minus}
\end{figure}

Figure \ref{fig:R4R4_06_c_minus} (left) shows the results of a particular run of the algorithm. 
The starting points are shown in blue and local minima are in red. The global minimum is denoted by a star.
The obtained $C_-$ is in a good agreement with the one obtained using $\dot{c}\in \left(\mbox{Lin}\{c,c_+,n\}\right)^{\bot}$ for the derivative along $C_-$ (shown in green).
The right panel of Figure \ref{fig:R4R4_06_c_minus} shows $z(c(t))$ for the two connected components of $C_-$.
One connected component has a loop topology, the other one is an interval with the end points with $\Rg Q(c)=n-2$. Numerical algorithm returns $z_{\rm max}=0.001059$ in this case.

\paragraph{Example 7. Artificial $\R^5\to \R^5$ map.}
{\em See file \underline{\tt examples/article\_example07.m}.}\\
\noindent The map is artificially-generated with all entries being integer, 
\begin{center}
$A_1=\left(\arraycolsep=1.4pt\def\arraystretch{1}
\begin{array}{ccccc}
-2 & 2 & 0 & -1 & 2 \\
2 & 0 & -1 & 0 & -2 \\
0 & -1 & -2 & 0 & 1 \\
-1 & 0 & 0 & -2 & -2 \\
2 & -2 & 1 & -2 & 2
\end{array}\right)$,
$A_2=\left(\arraycolsep=1.4pt\def\arraystretch{1}
\begin{array}{ccccc}
2 & 0 & 2 & 0 & -1 \\
0 & -2 & 0 & 0 & 0 \\
2 & 0 & -2 & -1 & 1 \\
0 & 0 & -1 & 0 & -2 \\
-1 & 0 & 1 & -2 & 0
\end{array}\right)$,
$A_3=\left(\arraycolsep=1.4pt\def\arraystretch{1}
\begin{array}{ccccc}
0 & -1 & -1 & 1 & 0 \\
-1 & 0 & -2 & -1 & 1 \\
-1 & -2 & 2 & 0 & 2 \\
1 & -1 & 0 & 0 & 1 \\
0 & 1 & 2 & 1 & 0
\end{array}\right)$,
$A_4=\left(\arraycolsep=1.4pt\def\arraystretch{1}
\begin{array}{ccccc}
-2 & 1 & 1 & 0 & 1 \\
1 & 0 & 2 & -1 & 1 \\
1 & 2 & -2 & 0 & 0 \\
0 & -1 & 0 & 0 & 0 \\
1 & 1 & 0 & 0 & 0
\end{array}\right)$,
$A_5=\left(\arraycolsep=1.4pt\def\arraystretch{1}
\begin{array}{ccccc}
5 & 2 & -1 & -1 & 2 \\
2 & 3 & -1 & 1 & 2 \\
-1 & -1 & 3 & -1 & -1 \\
-1 & 1 & -1 & 3 & 0 \\
2 & 2 & -1 & 0 & 3
\end{array}\right)$,\\
$b_1=\left(\arraycolsep=1.4pt\def\arraystretch{1}
\begin{array}{c}
-1 \\
1 \\
0 \\
0 \\
-1
\end{array}\right)$,
$b_2=\left(\arraycolsep=1.4pt\def\arraystretch{1}
\begin{array}{c}
0 \\
1 \\
-1 \\
1 \\
-1
\end{array}\right)$,
$b_3=\left(\arraycolsep=1.4pt\def\arraystretch{1}
\begin{array}{c}
1 \\
0 \\
0 \\
0 \\
1
\end{array}\right)$,
$b_4=\left(\arraycolsep=1.4pt\def\arraystretch{1}
\begin{array}{c}
1 \\
-1 \\
-1 \\
-1 \\
0
\end{array}\right)$,
$b_5=\left(\arraycolsep=1.4pt\def\arraystretch{1}
\begin{array}{c}
1 \\
1 \\
1 \\
-1 \\
1
\end{array}\right)$.
\end{center}
In this case $C_-$ is two-dimensional and we do not attempt to fully study it analytically. 
Rather we focus on the numerical tests of the proposed algorithms. First we test the boundary oracle starting from $y=0$ and $d$ shown below, which yields the distance to the boundary $t=0.1196$. 
Calling the algorithm to discover boundary non-convexities returns a non-trivial $c_-$,  certifying that the image is non-convex.

The corresponding map is definite. Using $c_+=(0.1326,-0.3859,0.1932,-0.6408,0.6209)$ and a default initial guess value $z^{\rm guess}_{\max}=137.5$ the algorithm 
performs $k=100$ iterations looking for vectors from $C_-$, identifies ten such vectors belonging to two  continuous components of $C_-$ and performs gradient descent yielding minimal $z_1=0.0935$ and $z_2=1.8862$ for each. The algorithm returns global minimum value $z_{\max}=0.0935$.

As a part of this example the algorithm also performs consistency check by generating random points $f(x)$ and asserting that they are correctly identified by {\tt infeasibility\_oracle}. 

\begin{center}
$d=\left(\arraycolsep=1.4pt\def\arraystretch{1}\begin{array}{c}
-1\\
-2\\
-3\\
-4\\
-5\\
\end{array}\right)$,
$c=\left(\arraycolsep=1.4pt\def\arraystretch{1}\begin{array}{c}
-0.0128\\
0.1989\\
0.1827\\
0.3844\\
0.8827
\end{array}\right)$,
$c_-=\left(\arraycolsep=1.4pt\def\arraystretch{1}\begin{array}{c}
-0.3136\\
0.1355\\
-0.1169\\
-0.3933\\
0.8456
\end{array}\right)$,
$c_+=\left(\arraycolsep=1.4pt\def\arraystretch{1}\begin{array}{c} 0.1326\\-0.3859\\ 0.1932\\ -0.6408\\ 0.6209\end{array}\right)$.
\end{center}
\paragraph{Example 8. Artificial $\mathbb{C}^3\to \R^5$ map.}
{\em See file \underline{\tt examples/article\_example08.m}.}\\ 
\noindent We study a $\mathbb{C}^3\to \R^5$ map
\begin{center}
	$A_1=\left(\arraycolsep=1.4pt\def\arraystretch{1}
	\begin{array}{ccc}
	-2 & 1 & 1\\
	1 & 2 & 1-i\\
	1 & 1+i & 2\\
	\end{array}\right)$,
	$A_2=\left(\arraycolsep=1.4pt\def\arraystretch{1}
	\begin{array}{ccc}
	-2 & -2 & 2+2i\\
	-2 & 2 & i\\
	2-2i & -i & 0\\
	\end{array}\right)$,
	$A_3=\left(\arraycolsep=1.4pt\def\arraystretch{1}
	\begin{array}{ccc}
	2 & -1-i & -1-2i\\
	-1+i & 0 & -1-i\\
	-1+2i & -1+i & -2\\
	\end{array}\right)$,
	$A_4=\left(\arraycolsep=1.4pt\def\arraystretch{1}
	\begin{array}{ccc}
	-2 & -1-2i & -i\\
	-1+2i & 0 & -1+i\\
	i & -1-i & 2\\
	\end{array}\right)$,
	$A_5=\left(\arraycolsep=1.4pt\def\arraystretch{1}
	\begin{array}{ccc}
	7 & -i & 0\\
	i & 5 & -i\\
	0 & i & 7\\
	\end{array}\right)$,
	$b_1=\left(\arraycolsep=1.4pt\def\arraystretch{1}
	\begin{array}{c}
	0\\
	0\\
	0\\
	\end{array}\right)$.
	$b_2=\left(\arraycolsep=1.4pt\def\arraystretch{1}
	\begin{array}{c}
	0\\
	-1\\
	1\\
	\end{array}\right)$.
	$b_3=\left(\arraycolsep=1.4pt\def\arraystretch{1}
	\begin{array}{c}
	1\\
	0\\
	1\\
	\end{array}\right)$.
	$b_4=\left(\arraycolsep=1.4pt\def\arraystretch{1}
	\begin{array}{c}
	0\\
	0\\
	1\\
	\end{array}\right)$.
	$b_5=\left(\arraycolsep=1.4pt\def\arraystretch{1}
	\begin{array}{c}
	-1\\
	-1\\
	1\\
	\end{array}\right)$.
	\end{center}
Running ${\tt get\_z\_max}$ with $z^{\rm guess}_{\max}=1$ and $k=300$ results in $z_{\max}=0.00768$.






\paragraph{Example 9. Artificial $\mathbb{C}^3\to \R^6$ map.}
{\em See file \underline{\tt examples/article\_example09.m}.}\\
\noindent We study an artificial $\mathbb{C}^3\to \R^6$ map
\begin{center}
	$A_1=\left(\arraycolsep=1.4pt\def\arraystretch{1}
	\begin{array}{ccc}
	-2 & 1 & 1\\
	1 & 2 & 1-1i\\
	1 & 1+1i & 2\\
	\end{array}\right)$,
	$A_2=\left(\arraycolsep=1.4pt\def\arraystretch{1}
	\begin{array}{ccc}
	-2 & -2 & 2+2i\\
	-2 & 2 & 1i\\
	2-2i & -1i & 0\\
	\end{array}\right)$,
	$A_3=\left(\arraycolsep=1.4pt\def\arraystretch{1}
	\begin{array}{ccc}
	2 & -1-1i & -1-2i\\
	-1+1i & 0 & -1-1i\\
	-1+2i & -1+1i & -2\\
	\end{array}\right)$,
	$A_4=\left(\arraycolsep=1.4pt\def\arraystretch{1}
	\begin{array}{ccc}
	-2 & -1-2i & -1i\\
	-1+2i & 0 & -1+1i\\
	1i & -1-1i & 2\\
	\end{array}\right)$,
	$A_5=\left(\arraycolsep=1.4pt\def\arraystretch{1}
	\begin{array}{ccc}
	2 & -1i & 0\\
	1i & 0 & -1i\\
	0 & 1i & 2\\
	\end{array}\right)$,
	$A_6=\left(\arraycolsep=1.4pt\def\arraystretch{1}
	\begin{array}{ccc}
	5 & -1i & 1\\
	1i & 3 & 1-2i\\
	1 & 1+2i & 7\\
	\end{array}\right)$,
	$b_1=\left(\arraycolsep=1.4pt\def\arraystretch{1}
	\begin{array}{c}
	0\\
	1\\
	0\\
	\end{array}\right)$.
	$b_2=\left(\arraycolsep=1.4pt\def\arraystretch{1}
	\begin{array}{c}
	-1\\
	-1\\
	0\\
	\end{array}\right)$.
	$b_3=\left(\arraycolsep=1.4pt\def\arraystretch{1}
	\begin{array}{c}
	0\\
	0\\
	1\\
	\end{array}\right)$.
	$b_4=\left(\arraycolsep=1.4pt\def\arraystretch{1}
	\begin{array}{c}
	-1\\
	-1\\
	1\\
	\end{array}\right)$.
	$b_5=\left(\arraycolsep=1.4pt\def\arraystretch{1}
	\begin{array}{c}
	1\\
	1\\
	-1\\
	\end{array}\right)$.
	$b_6=\left(\arraycolsep=1.4pt\def\arraystretch{1}
	\begin{array}{c}
	-1\\
	-1\\
	-1\\
	\end{array}\right)$.
	\end{center}
Starting with $z^{\rm guess}_{\max}=0.1$ and $k=100$, running ${\tt get\_z\_max}$ yields $z_{\max}=0.0335$.







\paragraph{Example 10. Homogeneous $\R^4\to \R^4$ map.}
{\em See file \underline{\tt examples/article\_example10.m}.}
\noindent We study an artificial homogeneous $\R^4\to \R^4$ map
\begin{center}
\label{example10}
	$A_1=\left(\arraycolsep=1.4pt\def\arraystretch{1}
	\begin{array}{cccc}
	0 & 1 & 0 & 0 \\
	1 & 0 & 0 & 0 \\
	0 & 0 & 0 & 1 \\
	0 & 0 & 1 & 0   
	
	\end{array}\right)$,
	$A_2=\left(\arraycolsep=1.4pt\def\arraystretch{1}
	\begin{array}{cccc}
	0 & 0 & 1 & 0 \\
	0 & 2 & 0 & 1 \\
	1 & 0 & 2 & 0 \\
	0 & 1 & 0 & 0   
	
	\end{array}\right)$,
	$A_3=\left(\arraycolsep=1.4pt\def\arraystretch{1}
	\begin{array}{cccc}
	0 & 0 & 0 & 1 \\
	0 & -1 & 1 & 0 \\
	0 & 1 & 1 & 0 \\
	1 & 0 & 0 & 0   
	
	\end{array}\right)$,
	$A_4=\left(\arraycolsep=1.4pt\def\arraystretch{1}
	\begin{array}{cccc}
	1 & 0 & 0 & 0 \\
	0 & 1 & 0 & 0 \\
	0 & 0 & 1 & 0 \\
	0 & 0 & 0 & 1   
	
	\end{array}\right)$,
\end{center}
and $b_i=0$.
This map is definite. We choose $c_+=(0,0,0,1)^T$, thus equating convexity of $F$ of this map with the convexity of joint numerical range of matrices $A_i$, $i=1,2,3$,
\begin{align}
\label{JNR}
y_i=x^T A_i x,\quad |x|^2=1.
\end{align}
Running {\tt nonconvexity\_certificate.m} confirms that $F$ is non-convex. The same can be established analytically, for example, by plotting the intersection of $F$ with the hyperplane $y_3=0$, see Figure~\ref{fig:example10}.

Since the image $F$ of a homogeneous map is a cone, the algorithm of section \ref{sec:convexsubpart} to identify a convex compact subpart of $F$ by "cutting" it with a hyperplane will not work. Hence the routine 
{\tt get\_z\_max} will return an exception in case matrix $b$ is zero or trivial in the sense of section \ref{Notations}. 


\begin{figure}[H]
\centering\includegraphics[width=8cm]{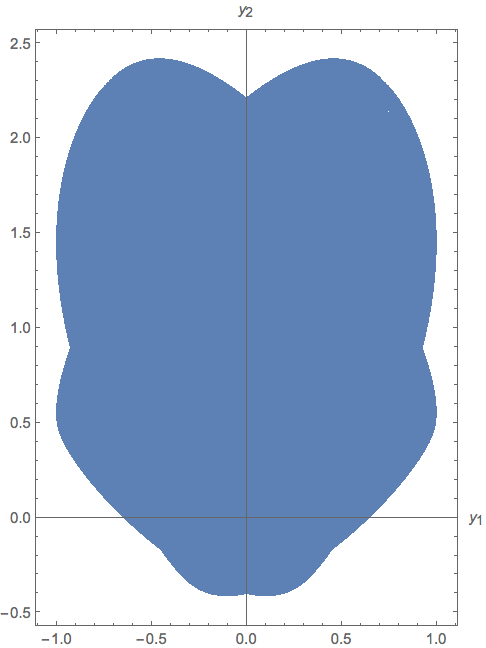}
\caption{The intersection of $F$, the image of the map \eqref{example10}, and the hyperplane $y_3=0$.}
\label{fig:example10}

\end{figure}





\section{Conclusion}
\label{conclusion}
In this paper we address a number of problems pertaining to the geometry of quadratic maps.
We consider general real and complex quadratic maps of the form \eqref{real} or \eqref{complex} and address the following tasks linked with the image of the map $F$. 
\begin{itemize}
\item Feasibility oracle: certifying that a given point (does not) belong to an image $F$ of a given quadratic map
\item Boundary oracle: finding a boundary point $y\in \partial F$ which lies on a given line
\item Convexity oracle: certifying that an image $F$ of a given quadratic map is non-convex
\item Convexity of a sub-region: finding a subregion of non-convex $F$ which is convex
\end{itemize}
From an algorithmic point of view these problems are not convex and some of them are known to be NP-hard. 
Our approach was to employ "hidden convexity" of quadratic maps, an observation that convex relaxation of various quadratic optimization problems often yields robust results. Hidden convexity allows us to reformulate feasibility and boundary oracles as standard problems of convex optimization \cite{PolyakGryazina17}. Another important observation is the result of \cite{Dymarsky_cuttingArX} that the image of quadratic map $F$ is convex if and only if it has no boundary non-convexities. Using this result we formulate convexity oracle and the problem of finding convex sub-region 
as the problem of finding boundary non-convexities. The latter problem can be efficiently addressed stochastically, yielding a finite probability of identifying boundary non-convexities, if any. 

In this paper we provide a detailed description of the proposed algorithms, together with the necessary mathematical foundations. The paper is accompanied by a MATLAB library CAQM (Convexity Analysis of Quadratic Maps), which implements the algorithms. Section \ref{num_res} of this paper contains an extensive discussion of ten numerical examples outlining functionality and efficiency of the library.


The MATLAB library CAQM is available at Github:
\href{http://github.com/sergeivolodin/CAQM}{github.com/sergeivolodin/CAQM}.

\section*{Acknowledgements}
The authors are thankful to Prof.~Janusz Bialek.
We gratefully acknowledge collaboration between the Institute for Control Sciences RAS and the Center for Energy Systems of Skolkovo Institute of Science and Technology.

\appendix

\section{Continuous case: gradient and normal}
\label{sec:gradient}
We consider a quadratic matrix $Q(t)$ smoothly depends on parameter and assume that $\dim\Ker(Q)=1$ for all $t$. By $x_0(t)$ we denote a normalized vector $x_0\in \Ker(Q)$ and $Q^{-1}$ stands for the pseudo-inverse
\begin{eqnarray}
\label{Qx}
Qx_0&=&0\ ,\\
Q^{-1}Q&=&QQ^{-1}={\mathbb I}-x_0 x_0^*\ .
\label{QQ}
\end{eqnarray}
After differentiating \eqref{Qx} and \eqref{QQ} by the parameter $t$ we find
\begin{eqnarray}
\label{xdot}
\dot{x}_0&=&-Q^{-1}Qx_0\ ,\\
\frac{d}{d t}Q^{-1}&=&-Q^{-1}\dot{Q}Q^{-1}+x_0 (x_0^*\dot{Q}Q^{-2})+(Q^{-2}\dot{Q}x_0) x_0^*\ .
\label{Qinvdot}
\end{eqnarray}

To make the connection with section \ref{contcase} we use $Q(t)=Q(c(t))$ and \eqref{q} to write
\begin{equation}
\dot{Q}=\dot{c}\cdot A-x_0^*(\dot{c}\cdot A)x_0 {\mathbb I}\ ,
\end{equation}
where we have used $\frac{d}{d t} \lambda_{\rm min}(c\cdot A)=x_0^*(\dot{c}\cdot A)x_0$, which follows from $(c\cdot A)x_0=\lambda_{\rm min}(c\cdot A)x_0$.

The condition $x_0^* (c\cdot b)=0$ after differentiating over $t$ and combining with \eqref{xdot} gives the expression for $n_i$ \eqref{q},
\begin{equation}
\sum_i \dot{c}_i (x_0^* q_i)=0\ ,\quad q_i=b_i -(A_i-(x_0^* A_i x_0){\mathbb I})Q(c)^{-1}(c\cdot b)\ .
\end{equation}

Similarly, after differentiating $z(t)=z(c(t))$ with respect to $t$ and using \eqref{Qx} and $Q^{-1}x_0=0$, as well as \eqref{Qinvdot}, we obtain
\begin{equation}
\dot{z}=2\Re(v^*Q^{-1}(\dot{c}\cdot b-\dot{Q}v))
\end{equation}
From here it follows
\begin{equation} \frac{\partial z}{\partial c_i}=2\Re(v^*Q^{-1}q_i)\ . \end{equation}

\section{Boundary non-convexities in homogeneous case}
\label{non_convex_homogeneous}

In this section the quadratic map $f\colon\mathbb{V}^n\to\mathbb{R}^m$
is homogeneous, meaning that the linear part in the definition (\ref{real}) or (\ref{complex}) is zero:
\begin{equation}
b_k\equiv 0\in\mathbb{V}^n,\, k=1...m
\end{equation}
Thus, the map f has the form:
\begin{equation}
f_k(x)=x^* A_k x,\,A_k^*=A_k,\,k=1..m
\end{equation}

By reasons mentioned in the Section \ref{boundary_nonconvexity}, the Proposition \ref{th:noconv_cert} is not applicable for homogeneous case.
This gives rise to a new

\begin{proposition}[Sufficient condition for non-convexity of $\partial F_c$ in homogeneous case]
\label{th:noconv_cert_homogeneous}
If for a homogeneous quadratic map $f$ with $m\geqslant 3,\,n\geqslant 2$ and some $c$, matrix $c\cdot A$ is singular, positive semi-definite, and $\Ker(c\cdot A)$ is 2-dimensional with a basis $x_0,x_1$, moreover, vectors $u_k=x_0^* A_k x_0$, $v_k=x_1^* A_k x_1$, $w_k=x_0^* A_k x_1$ are linearly independent, then
$$
\partial F_c=f(\Ker(c\cdot A))=\{f(x)\,\big|\,x=t_0 x_0+t_1x_1,\,t_0,\,t_1\in\Cc\} \mbox{ is non-convex.}
$$

As in the Proposition \ref{th:noconv_cert} for a non-homogeneous map, in case of a complex map, vectors $\Re w$ and $\Im w$ should be considered instead of just one vector $w$.
\end{proposition}
\begin{proof}
Consider a point $f(x)$ from $\partial F_c$: $f(t_0 x_0+t_1x_1)=r_0^2u+r_1^2v+2r_0r_1\cos\varphi \Re w-2r_0r_1\sin\varphi\Im w$, where $r_i=|t_i|,\,i\in0,\,1$, $t_0^* t_1=r_0r_1e^{i\varphi}$.
Obviously, $u$ and $v$ belong to the $\partial F_c$ (take $t_0=1$, $t_1=0$ for $u$).
Assuming that $\partial F_c$ is convex, the point $\frac{u+v}{2}$ should also belong to $\partial F_c$.
Then for some $r_i$ and $\varphi$,
$$
\left(r_0^2-\frac{1}{2}\right)u+\left(r_1^2-\frac{1}{2}\right)v+2r_0r_1\cos\varphi \Re w-2r_0r_1\sin\varphi\Im w=0
$$
Since vectors $u,\,v,\,\Re w,\,\Im w$ are linearly independent, all the coefficients should be equal to zero.
This leads to a contradiction: $\cos^2\varphi+\sin^2\varphi=0$ \lightning.
Thus, $\partial F_c$ is non-convex.
\end{proof}

In the homogeneous case, the set of non-convexities used in the numerical algorithm (\ref{cminorig}) is defined in a different way, according to the corresponding Proposition \ref{th:noconv_cert_homogeneous}:
\begin{equation}
C_-=\{c\in\mathbb{R}^m\,\big|\,c\cdot A\succeq 0,\,\dim(\Ker(c\cdot A))=2\}\label{cminusnew}
\end{equation}

The condition on the vectors $u,\,v,\,w$ in Proposition \ref{th:noconv_cert_homogeneous} being linearly dependent is deliberately not considered in the new definition (\ref{cminusnew}) of $C_-$ for the same reason as the corresponding similar condition is not considered in the non-homogeneous case, namely, because such a case is rare, as it is argued in the Section \ref{boundary_nonconvexity}.

\end{document}